\documentclass{amsart}
\usepackage[utf8]{inputenc}
\usepackage{amsmath}
\usepackage{amsfonts}
\usepackage{amssymb}
\usepackage{amsthm}
\usepackage{mathrsfs}
\usepackage{graphicx}
\usepackage[dvipsnames,svgnames,table]{xcolor}
\graphicspath{{./figures/}}
\usepackage{hyperref}
\hypersetup{colorlinks,
		filecolor=black,	
		linkcolor=blue,
		citecolor=darkgray,
		urlcolor=RoyalBlue,
		bookmarksopen=true}
\usepackage{datetime}
\usepackage{enumerate}

\newcommand{\mb}{\mathbb}
\newcommand{\mc}{\mathcal}
\newcommand{\ms}{\mathscr} % requires usepackage{mathrsfs}
\newcommand{\mr}{\mathrm}
\newcommand{\mf}{\mathfrak}
\newcommand{\ope}{\mathscr} % for operators
\newcommand{\C}{{\mathbb C}}
\newcommand{\R}{{\mathbb R}}
\newcommand{\N}{{\mathbb N}}
\newcommand{\Z}{{\mathbb Z}}
\newcommand{\D}{\nabla}
\newcommand{\cv}{\D}
\newcommand{\TSM}{T^*\mfd }
\newcommand{\sym}{S^{2,0}}  % until we think of better notation

\newcommand{\ds}{\mr{d}s}
\newcommand{\dx}{\mr{d}x}
\newcommand{\Ds}{\partial_s}
\newcommand{\ps}{\frac{\partial}{\partial s}}
\newcommand{\pd}{\partial}
\newcommand{\lp}{\langle}
\newcommand{\rp}{\rangle}

\newcommand{\ve}{\varepsilon}
\newcommand{\vp}{\varphi}
\newcommand{\mat}{\begin{pmatrix}}
\newcommand{\rix}{\end{pmatrix}}
\newcommand{\smat}{\left(\begin{smallmatrix}}
\newcommand{\srix}{\end{smallmatrix}\right)}
\newcommand{\rfc}{\mf{rfc}}
\renewcommand{\Re}{\mathop{\mathrm{Re}}}
\newcommand{\gh}{\tilde g}
\newcommand{\ball}{{\mc B}}
\newcommand{\mfd}{{\mc M}}

\DeclareMathOperator{\Rc}{Rc}
\DeclareMathOperator{\Rm}{Rm}

\DeclareMathOperator{\resolvent}{\mathscr{R}}

\newtheorem{theorem}{Theorem}  %  One counter to count them all
\newtheorem{lemma}[theorem]{Lemma}

\newtheorem{definition}[theorem]{Definition}
\newtheorem{prop}[theorem]{Proposition}
\newtheorem{main}{Main Theorem}	 	

\numberwithin{theorem}{section}
\numberwithin{equation}{section}

\begin{document}

\title	[Infinite-dimensional dynamical instabilities]
	{Infinite-dimensional dynamical instabilities of noncompact stationary Ricci flow solutions}
  
\author{Sigurd B. Angenent}
\author{Dan Knopf}

\begin{abstract}
Regarding Ricci flow as a dynamical system, we derive sufficient conditions for noncompact stationary (Ricci-flat) solutions to
possess infinite-dimensional unstable manifolds, and provide examples satisfying those criteria that have uncountably many
unstable perturbations.
\end{abstract}

\maketitle

\begingroup\footnotesize\sffamily
\tableofcontents
\endgroup

\section{Introduction}~

In this work, we develop a set of tools for investigating the stability or instability of noncompact Ricci-flat manifolds. We are
primarily interested in formulating criteria for instability of noncompact solitons that are asymptotically conical. We prove
results that apply to general noncompact manifolds with bounded geometry and then demonstrate their applicability to noncompact
Ricci-flat manifolds which form a subset of a family of Einstein metrics discovered by Christoph Böhm~\cite{B99}.

\subsection{Main results}

More precisely, our conclusions are as follows. Let $(\mfd , g) $ be a complete Ricci-flat manifold with bounded geometry.  We
introduce a geometric condition that implies that the Lichnerowicz Laplacian acting on symmetric 2-tensors,
\begin{equation} \label{Lich-Lap}
  \Delta_\ell h_{jk}\stackrel{\rm def}=g^{pq}\D_p\D_q h_{jk} + 2 R^q_{pjk}h_q^p - R_j^p h_{pk}-R_k^q h_{jq},
\end{equation}
generates analytic semigroups on a variety of spaces, including $L^2(T^*\mfd \otimes_{\rm S}T^*\mfd )$ and the
little-Hölder\footnote{~We write $h^{k,\alpha}(V)$ for the so-called \emph{little Hölder spaces} of sections of a vector bundle $V\to\mfd$.
  These may be defined as the closure of smooth sections, \emph{i.e.,} $C^\infty(V)$, in the usual Hölder space $C^{k,\alpha}$.  See
  Appendix~\ref{About little Holder} for a short review of these spaces.  } spaces $h^{0,\alpha}(T^*\mfd \otimes_{\rm S}T^*\mfd )$.
Moreover, we show that the essential spectrum of $\Delta_\ell$ is exactly the negative real axis,
$ \sigma_{\rm ess}(\Delta_\ell) = (-\infty, 0]$.  This implies that any remaining spectrum consists of isolated eigenvalues, which may
accumulate only on the essential spectrum. It is worth noting that this is true regardless of which function space we choose, which may
not be the case for manifolds that are not asymptotically flat.\footnote{~While the spectrum of $\Delta_\ell$ turns out to be the same
  in both cases, this is in general not true.  For example, the spectrum of the Laplacian on the Poincaré disk $\mb D^2$ viewed as an
  unbounded operator in $L^p(\mb D^2)$ depends on $p$.  In particular, the $L^2$ spectrum is the half-line $(-\infty, -\frac{1}{4}]$,
  while the spectrum in $L^\infty$ is the region $\{x+iy \in\C \mid x\leq -y^2\}$: this region contains an open subset of the complex
  plane in spite of the fact that $\Delta_{\mb D^2}$ is formally self adjoint.  See \cite{MR689073,MR1016445}.  In \cite{MR3552794}
  Ammann and Große study more general (Dirac) operators and also provide a useful history and references.

  We note that the fact that the spectrum of $\Delta_\ell$ on the Böhm soliton $(\mfd, g)$ is the same in $L^2$ and $h^{0,\alpha}$ turns
  out to be a consequence of the asymptotic flatness of $(\mfd, g)$.}
  \medskip

A geometric criterion that we assume throughout this paper is as follows:
\begin{definition}[Bounded geometry]
One says a complete Riemannian manifold $(\mfd , g)$ has \emph{bounded geometry} if there exists $C>0$ such that:
\begin{itemize}
\item the Riemann tensor is uniformly bounded, $\sup_{\mfd } |\Rm| \leq C$, and
\item the injectivity radius is bounded from below, $\mr{inj}(g)\geq1/C$.
\end{itemize}
\end{definition}
For an overview of some of the consequences of this assumption, the reader may wish to consult~\cite{TAubin98} or~\cite{Hebey}.

In addition, we will frequently assume the our manifold is asymptotically flat, in the following sense.

\begin{definition}[Asymptotically flat] \label{Asymptotically flat} A Riemannian manifold $(\mfd, g)$ of bounded geometry is
\emph{asymptotically flat} if there exists a proper function $r:\mfd \to (0, \infty)$ such that the exponential map
$\exp_p:T_p\mfd \to \mfd $ is injective on
\[
\ball_{r(p)}(p) = \{v\in T_p\mfd \colon |v|\leq r(p)\}.
\]
Furthermore, as $p\to\infty$, the pullback metrics $g_p = \exp_p^*(g)$ on $\ball_{r(p)}(p)$ converge in $C^{2,\alpha}$ to the flat
Euclidean metric on $\ball_{r(p)}\subset T_p\mfd $.
\end{definition}

We note that asymptotically conical metrics --- which constitute our main examples in this paper --- are asymptotically flat.  \medskip

Our first principal result concerns ancient solutions to a Ricci-DeTurck flow: after choosing a DeTurck vector field $X$ --- see
\eqref{eqn:slick X}) ---  we consider the initial value problem for
\begin{equation}		\label{RDT with X}
\frac{\partial \tilde g}{\partial t} = - 2\Rc[\tilde g] + \mc L_X \tilde g,
\end{equation}
whose solutions are equivalent to solutions of the Ricci flow $\frac{\partial g}{\partial t} = -2\Rc[g]$.

\goodbreak%There was an awkward page break splitting Main Theorem I in two parts
\begin{main}		\label{MT1}
Let $(\mfd, g)$ be an asymptotically flat Ricci-flat manifold.

For any $\delta>0$, the set $\{\lambda\in\sigma(\Delta_\ell) : \Re\lambda > \delta\}$ consists of a finite number (possibly zero) of
real eigenvalues
\[
\lambda_1\geq \lambda_2 \geq \lambda_3 \geq \cdots \geq \lambda_{N_\delta}.
\]
If $N_\delta>0$, then there exists a real analytic $N_\delta$-dimensional family,
\[
\ms M\subset h^{2,\alpha}(T^*\mfd \otimes_{\rm S}T^*\mfd ),
\]
such that each metric $g(0)\in \ms M$ belongs to an ancient solution
\[
\{g(t):-\infty<t\leq 0\}\subset\ms M
\]
that converges exponentially to $g$ in backward time.
\end{main}
\medskip

\subsubsection*{Group invariance}
If \(\mf G\) is a compact subgroup of the group of isometries of the Ricci-flat manifold \((\mfd, g)\), then our discussion of ancient
solutions still applies when one considers \(\mf G\)-invariant metrics on \(\mfd\) near \(g\).  In this case one must also only
consider \(\mf G\)-invariant eigentensors of the Lichnerowicz Laplacian and their corresponding eigenvalues.  The conclusion of the
Main Theorem then is that there exist a finite number of eigenvalues
\(\lambda_1^{\mf G}\geq \cdots\geq \lambda_{N_\delta}^{\mf G}>\delta\) whose eigentensors are \(\mf G\)-invariant, and that there
exists a real analytic \(N_\delta\) dimensional family \(\ms M_{\mf G}\subset h^{2, \alpha}\) consisting of \(\mf G\)-invariant metrics
\(g(0)\), each of which lies on a \(\mf G\)-invariant ancient solution \(\{g(t) : -\infty<t\leq 0\} \subset \ms M\).

\bigskip

Our second principal result shows that the hypothesis $N_\delta>0$ is not vacuous. Before stating it, we recall a special case of a
cohomogeneity-one construction of Einstein manifolds developed by Christoph Böhm --- please see Theorems~6.1 and 6.2 of~\cite{B99} for
statements in full generality\footnote{~In particular, Böhm shows that for given integers $r\geq0$ and $k,k_1,\dots,k_{r+1}\geq2$ there
  exists an $r$-dimensional family of Ricci-flat metrics on $ \R^{k+1}\times\mc S^{k_1}\times\cdots\times\mc S^{k_{r+1}} $, and even on
  some quotients of these spaces.}

\begin{prop}[Böhm~\cite{B99}]		\label{prop:Bohm exists}
Given integers \( n_1, n_2 \geq 2\), there exists a Ricci flat metric on $\R^{n_1+1}\times \mc S^{n_2}$.
\end{prop}

The Böhm metrics are invariant under the action of
\[
\mf G \stackrel{\rm def}= \mr {O}(n_1+1)\times\mr{O}(n_2+1)\quad\mbox{ on }\quad \R^{n_1+1}\times \mc S^{n_2};
\]
they have a multiply warped product structure, and are asymptotic to Ricci-flat cones at spatial infinity.  The group $\mf G$ is the
full isometry group for the Böhm metrics in Proposition~\ref{prop:Bohm exists}.  These metrics also play a prominent role in our recent
work~\cite{GAFA}.

To apply Theorem~\ref{MT1} to the Böhm metrics, we study the spectral properties of their Lichnerowicz Laplacians, both in
$L^2$ and in $h^{0,\alpha}$ spaces.  
The group $\mf G$ acts on \( \mfd \) and therefore also acts on $(2,0)$-tensor fields by pullback
\[
(\psi, h)\in \mf G \times L^2(\sym\mfd) \mapsto \psi^*h \in L^2(\sym\mfd).
\]
Since the action of any $\psi\in\mf G$ on $\mfd$ is smooth, its action on $(2,0)$-tensor fields preserves regularity (\emph{i.e.,}
$h^{0,\alpha}$ tensor fields are mapped to $h^{0,\alpha}$-tensor fields, \emph{etc.}).  Any \(\psi\in\mf G\) also acts on \(\mfd\) by
isometries, and therefore the Lichnerowicz Laplacian~\(\Delta_\ell\) commutes with \(\psi^*\).  This implies that \(h\mapsto \psi^* h\)
maps eigentensors to eigentensors with the same eigenvalue.  In particular, we can consider eigentensors that are invariant under the
action of \(\mf G\).

\medskip
We obtain the following.

\begin{main}		\label{MT2} 
For any integers $n_1, n_2\geq2$ with $n_1+n_2\leq8$, consider the Ricci-flat metric $g$ on
$\mfd=\R^{n_1+1}\times\mc S^{n_2}$ produced by Böhm's construction.
\begin{enumerate}[\upshape (a)]
\item The Lichnerowicz Laplacian $\Delta_\ell$ on $(\mfd, g)$ acting in $L^2(\sym\mfd)$ is self adjoint.
\item The Lichnerowicz Laplacian $\Delta_\ell$ on $(\mfd, g)$ acting in $h^{0,\alpha}(\sym\mfd)$ generates an analytic semigroup.
\item The spectrum of $\Delta_\ell$ is the same in $L^2(\sym\mfd)$ as in $h^{0,\alpha}(\sym\mfd)$.
\item The essential spectrum of $\Delta_\ell$ is $(-\infty, 0]$.
\item The point spectrum of $\Delta_\ell$ consists of an infinite sequence of positive eigenvalues
\[
\lambda_0\geq \lambda_1\geq \lambda_2\geq \cdots \to 0.
\]
Their corresponding eigentensors decay exponentially.
\item The eigenvalues of \(\Delta_\ell\) corresponding to \(\mf G\)-invariant eigentensors are simple, and there are still infinitely
many of these.
\end{enumerate}
Finally, for any $N\in\N$, the $N$-dimensional family $\ms M$ of ancient solutions corresponding to $\lambda_0, \dots, \lambda_{N-1}$ produced
by Main Theorem~\ref{MT1} contains uncountably many solutions of Ricci--DeTurck flow that are pairwise geometrically distinct for as long as they exist.
\end{main}

We explain the dimension restriction $n_1+n_2\leq 8$ in \S~\ref{Lichnerowicz} below.  The last statement in the theorem addresses the
issue that the ancient solutions provided by Main Theorem~\ref{MT1} are solutions of the Ricci-DeTurck flow, and that it is in
principle possible that they all come about by applying different diffeomorphisms to one and the same ancient solution. We prove that this cannot happen.

\medskip

There is a vast literature on dynamic and linear stability of Ricci flow, sometimes considering the second variation of a Perelman
functional and sometimes utilizing the spectrum of the Lichnerowicz Laplacian. Results concerning the stability of some noncompact
solutions may be found in~\cite{SSS08} and~\cite{SSS11}.  Results about the instability of various noncompact solutions may be found
in~\cite{GiHaPo2003} and~\cite{DO23}. This list is very, very far from inclusive. For example, see~\cite{MR4534486} and the
recent~\cite{deruelle2025orbifoldsingularityformationancient}.
    
\subsection{Outline of the paper}

Main Theorem~\ref{MT1} is proved in \S~\ref{sec:DeTurck}--\ref{Spectrum of LL}.  In \S~\ref{sec:DeTurck}, we review DeTurck's
derivation of a quasilinear strictly parabolic \textsc{pde} equivalent to Ricci flow, in a form adapted to our subsequent needs. In
\S~\ref{sec:AnalyticSemigroups}, we explain how the linearization of DeTurck flow generates an analytic semigroup, a claim which we
prove in \S~\ref{sec:L generates}, and we review results of Chaperon~\cite{Chaperon2002, Chaperon2004} and
Irwin~\cite{IrwinProofPseudoStableMfdThm,IrwinBook} which motivate the approach used in this paper. In \S~\ref{Spectrum of LL}, we
investigate the essential spectrum of the Lichnerowicz Laplacian on the Ricci-flat manifolds we study.

Main Theorem~\ref{MT2} is proved in \S~\ref{Lichnerowicz}--\ref{sec:Distinct}.  In \S~\ref{Lichnerowicz}, we demonstrate that the
Böhm stationary solutions have infinitely many unstable eigenvalues, and that those with $\mf G$-invariant eigentensors are simple.
Then in \S~\ref{sec:Distinct}, we show that these unstable perturbations of the Böhm solutions yield uncountably many geometrically distinct solutions of Ricci flow.

In Appendices \S~\ref{Geometric data} and \S~\ref{Tensor Hessian}, we derive various standard facts about doubly warped product
geometries like the Böhm solutions. In Appendix~\ref{sec:invariant manifold construction}, we outline a proof of the unstable
manifold construction, Theorem~\ref{UnstableManifoldTheorem}, used in this work. In Appendix~\ref{OscillationBound}, we prove
that oscillations in the metric components of any solution originating from an unstable pertubation of the B\"ohm solutions we consider
cannot increase in time. Finally, in Appendix \ref{sec:function spaces}, for completeness, we briefly review the main function spaces used herein.

\section{An approach to DeTurck's trick} \label{sec:DeTurck}

Here, we derive the quasilinear heat equation that results from DeTurck's trick (See~\cite{MR0697987}).

\subsection{Perturbed metrics and connections}

Let $g$ be a fixed background metric on a manifold $\mfd $, and consider a perturbation of $g$ given by
\[
  \gh_{ij} = g_{ij}+h_{ij}.
\]
Except as noted below, we raise and lower indices with $g$, so $\gh^{ij} = g^{ip}g^{jq} \gh_{pq}$ is in general not the inverse of
$\gh_{ij}$.  Because of this, we write the inverse of $\gh_{ij}$ as $G^{ij}$, \emph{i.e.,} by definition, we have
\[
  \gh_{ik}G^{kj} = \delta_i^j.
\]
For small $h$, the inverse is given by the geometric series
\begin{equation}		\label{geometric series for Gij}
  G^{ij} = g^{ij} + \sum_{\ell = 1}^\infty (-1)^{\ell}\, g^{ik_1} h^{k_2}_{k_1}h^{k_3}_{k_2}\cdots h^j_{k_\ell} .
\end{equation}
This series converges at any point $p\in\mfd$ such that $\|\hat h(p)\|<1$, where $\|\hat h(p)\|$ is the operator norm of the linear
transformation $\hat h^i_j = g^{ik}h_{kj}$ on the vector space $T_p\mfd$ with inner product $g(p)$.  The convergence is uniform if
$\|\hat h\| = \sup_{p\in\mfd}\|\hat h(p)\|<1$.  If, in addition, the covariant derivatives $\nabla^j\hat h$ of order $j\leq N$ are
uniformly $\alpha$-Hölder continuous, then the series~\eqref{geometric series for Gij} converges in $C^{N, \alpha}$.
\medskip

We denote the Levi-Civita connections of $g$ and $\gh$ by $\D$ and $\tilde\nabla$, respectively.

\subsubsection*{Generic lower order terms}
In this section, it will be convenient to let $Q=Q_{jk}\,\mr dx^j\,\mr dx^k$ stand for a generic quadratic polynomial tensor in
$\D h$ of the form
\begin{equation}
  \label{eq:quadratic-in-grad-h}
  Q = Q_0(g, h) + Q_1(g, h)*\D h + Q_2(g, h)*(\D h\otimes\D h),
\end{equation}
whose coefficients $Q_0, Q_1, Q_2$ are rational functions of $g, h$ defined for $\|h\|<1$ that satisfy
\[
  Q_0(h) = \mc O(\|h\|^2), \quad Q_1(h) = \mc O(\|h\|), \quad Q_2(h) = \mc O(1)\qquad \text{as}\quad \|h\| \to 0.
\]
Below, we allow $Q$ to change from line to line.  \medskip

Our main result in this section is the following version of DeTurck's trick:

\begin{lemma}\label{thm:DeTurck}
  There exists a tensor field $Q_{jk}(h, \D h)$ such that, if $\gh=g +h$ for a fixed background metric $g$, the tensor field $h$ is
  a solution of
  \begin{equation} \label{eq:} h_t = \ope F[h]
  \end{equation}
  with
  \begin{multline}\label{eq:DeT}
    \ope F_{jk}[h] = -2 R_{jk} +\Big\{G^{\ell m}(h)\D_\ell \D_m h_{jk}
    + 2 R_{\ell jk}^m h_m^\ell  - R_j^\ell h_{\ell k}  -  R_k^m h_{jm}\Big\}\\
    + Q_{jk}(h, \D h)
  \end{multline}
  if and only if $\gh_{jk}$ satisfies
  \begin{equation}\label{eq:RDT}
    \partial_t \gh_{jk} = -2\tilde R_{jk}+(\mc L_X \gh)_{jk}.
  \end{equation}
  
  The vector field $X$ is defined as follows: if $\gamma$ is the tensor field representing the difference of the Levi--Civita
  connections, \emph{i.e.,} if $\tilde\nabla_i W^k - \D_i W^k = \gamma_{ij}^k W^j$ for all vector fields $W$, then $X$ is the
  contraction\footnote{ The reader may readily verify that this definition is equivalent to the traditional formula,
    \[
      X^k = G^{ij}(h)G^{k\ell}(h)\Big(\cv_i h_{j\ell}-\frac12\cv_\ell h_{ij}\Big).
  \]}
  \begin{equation} \label{eqn:slick X} X^k \stackrel{\rm def}= G^{ij}\gamma_{ij}^k.
  \end{equation}
\end{lemma}

% {\teal More importantly, in our earlier notes, we traced with $g^{-1}$ instead of $\tilde g^{-1}$.}\medskip

We proceed to the derivation of these identities and the proof of the Lemma.

\subsection{The Levi-Civita connections}
As noted above, the Levi-Civita connections corresponding to $\gh$ and $g$ are $\tilde\nabla$ and $\D$, respectively.  The tensor
field $\gamma$ is given by
\begin{equation}
\gamma_{ij}^k =  \frac 12 G^{k\ell}\bigl\{ \D_i h_{j\ell}+\D_j h_{i\ell} - \D_\ell h_{ij}\bigr\}.
\end{equation}
A trace simplifies to
\begin{equation} \label{eq:D log det M} \gamma_{ij}^i = \frac 12 G^{im}\D_j \gh_{im} =\frac 12 \D_j \log\det\bigl(I + \hat h\bigr),
\end{equation}
where $\hat h = g^{-1}\circ h$, as above.

\subsection{Lie derivative of the perturbed metric}
If $W^i$ is any vector field, then the Lie derivative of the metric $\tilde g_{jk}$ along $W$ is given by
\begin{equation}\label{eq:3}
\mc L_W \gh_{jk} =\tilde \nabla_jW_k+ \tilde\nabla_k W_j
= \D_jW_k + \D_kW_j + 2\gamma_{jk}^\ell  W_\ell ,
\end{equation}
where $W_j = \gh_{ij}W^i$. This is the one instance where we use the perturbed metric $\gh_{ij}$ to lower an index.

\subsection{Perturbed curvature tensors}
The Riemann tensors of the background and perturbed metrics are related by
\begin{equation}
\tilde R_{ijk}^\ell  = R_{ijk}^\ell 
+ \D_i\gamma_{jk}^\ell  - \D_j\gamma_{ik}^\ell 
+ \gamma_{im}^\ell \gamma_{jk}^m - \gamma_{jm}^\ell \gamma_{ik}^m.
\end{equation}
For the Ricci tensors, we have
\[
\tilde R_{jk} = R_{jk} + \D_\ell \gamma_{jk}^\ell - \D_j\gamma_{\ell k}^\ell + \gamma_{\ell m}^\ell \gamma_{jk}^m - \gamma_{jm}^\ell
\gamma_{\ell k}^m.
\]
The ``$\gamma*\gamma$'' terms are exactly of the form $Q_{jk}$ introduced above, so we have
\begin{equation} \label{eq:1}
-2\tilde R_{jk} = -2 \D_\ell \gamma^\ell _{jk} + 2\D_j\gamma^\ell _{\ell k} + Q_{jk}-2R_{jk}.
\end{equation}

The first term on the right in~\eqref{eq:1} is a divergence and can be rewritten as
\begin{align*}
  -2\D_\ell \gamma_{jk}^\ell 
  &= - \D_\ell \Bigl\{G^{\ell m}\bigl(\D_j h_{km}+\D_k h_{jm} - \D_m h_{jk}\bigr) \Bigr\}\\
  &=  -G^{\ell m}\bigl(\D_\ell \D_j h_{km}+\D_\ell \D_k h_{jm} - \D_\ell \D_m h_{jk}\bigr)  + Q_{jk}\\
    % &=  -G^{\ell m}\D_\ell \bigl(\D_j h_{km}+\D_k h_{jm} - \D_m h_{jk}\bigr)  + Q_{jk}  \\
  &= G^{\ell m}\D_\ell \D_m h_{jk} -  G^{\ell m}\bigl(\D_\ell \D_j h_{km}+\D_\ell \D_k h_{jm} \bigr)  + Q_{jk}.
\end{align*}
The first term in the last line above is an elliptic operator applied to $h_{jk}$.

We can rewrite the second group of terms in the last line above as the Lie derivative of $g$ plus lower order terms by switching the order of
derivatives.  Namely, we compute that
\begin{align*}
G^{\ell m}\D_\ell \D_j h_{km}
&= G^{\ell m}\bigl\{\D_j\D_\ell  h_{km}  - R_{\ell jk}^p h_{pm} - R_{\ell jm}^p h_{kp} \bigr\} \\
&= \D_j\bigl(G^{\ell m}\D_\ell  h_{km}\bigr)  - G^{\ell m}R_{\ell jk}^p h_{pm} - G^{\ell m}R_{\ell jm}^p h_{kp}   + Q_{jk}\\
&= \D_j\bigl(G^{\ell m}\D_\ell  h_{km}\bigr)  - g^{\ell m}R_{\ell jk}^p h_{pm} - g^{\ell m}R_{\ell jm}^p h_{kp}   + Q_{jk}\\
&= \D_j\bigl(G^{\ell m}\D_\ell  h_{km}\bigr)  - R_{\ell jk}^p h_p^\ell  + R_j^p h_{kp}   + Q_{jk} ,
\end{align*}
and hence obtain
\[
G^{\ell m}\bigl(\D_\ell \D_j h_{km}+\D_\ell \D_k h_{jm} \bigr)
=
\D_j V_k + \D_k V_j - 2R_{\ell jk}^p h_p^\ell  + R_j^ph_{kp} + R_k^p h_{jp} + Q_{jk},
\]
where $V$ is the covector
\[
V_k \stackrel{\rm def}= G^{\ell m}\D_\ell h_{mk}.
\]
Using~\eqref{eq:3}, we then have
\begin{multline*}
G^{\ell m}\bigl(\D_\ell \D_j h_{km}+\D_\ell \D_k h_{jm} \bigr)
=
\mc L_V \gh_{jk} + 2\gamma_{jk}^\ell V_\ell  \\
- 2R_{\ell jk}^p h_p^\ell  + R_j^ph_{kp} + R_k^p h_{jp} + Q_{jk}.
\end{multline*}
Since $V_i=G^{\ell m}\D_\ell h_{mi}$, the term $\gamma_{jk}^\ell  V_\ell $ may be absorbed into $Q_{jk}$, so we find that
\begin{equation}
G^{\ell m}\bigl(\D_\ell \D_j h_{km}+\D_\ell \D_k h_{jm} \bigr)
=
\mc L_V \gh_{jk} 
- 2R_{\ell jk}^p h_p^\ell  + R_j^ph_{kp} + R_k^p h_{jp} + Q_{jk}.
\end{equation}
\medskip

Now by~\eqref{eq:D log det M}, the second term on the right in \eqref{eq:1} is
\[
2\D_j\gamma_{\ell k}^\ell  = \D_j\D_k \log \det\bigl(I+\hat h\bigr) = \D_j U_k + \D_k U_j,
\]
where $U$ is the covector
\[
U_k \stackrel{\rm def}= \frac{1}{2} \D_k \log \det \bigl(I+\hat h\bigr).
\]
Therefore, we collect terms and conclude that
\begin{multline*}
-2R_{jk} + \mc L_V \gh_{jk} - \mc L_U g_{jk} = G^{\ell m}\D_\ell \D_mh_{jk}   + 2R_{\ell jk}^p h_p^\ell  - R_j^ph_{kp}  - R_k^p h_{jp}
-2R_{jk} + Q_{jk}.
\end{multline*}
Choosing $X=(V-U)^\sharp$, namely,
\[
X^k =G^{ij}\gamma_{ij}^k = G^{ij}G^{k\ell}\Big(\cv_i h_{jk}-\frac12\cv_\ell h_{ij}\Big),
\]
and noting that $\mc L_X g_{jk} = \mc L_X \gh_{jk} + Q_{jk}$, we complete the proof of Lemma~\ref{thm:DeTurck}.

\section{Short time existence and smooth well-posedness}	\label{sec:AnalyticSemigroups}

\subsection{Analytic semigroups}

We write
\[
\sym \mfd \stackrel{\rm def}= \TSM\otimes_{\rm sym}\TSM
\]
for the bundle of symmetric $(2,0)$ tensors on $\mfd $.
\smallskip

We assume in this section that $(\mfd, g)$ has bounded geometry.

\begin{theorem}		\label{lem:an-sgp}
Let \(\alpha\in(0,1)\), and \(A^{pq}\), \(B^{pqr}_{jk}\), and \(C^{pq}_{jk}\) be tensor fields on \(\mfd \) with \(A\in C^1\),
and \(B,C\in h^{0,\alpha}\).  Assume that \(A^{pq}\) is uniformly elliptic, and that the gradients \(\nabla A^{pq}, \nabla B^{pqr}_{jk}\) are
uniformly bounded.  Then the linear operator
\[
\ope L \colon h^{2,\alpha}(\sym\mfd ) \to h^{0,\alpha}(\sym\mfd )
\]
defined by
\[
\ope L h_{jk} = A^{pq}\D_p \D_q h_{jk} + B^{pqr}_{jk} \D_p h_{qr} + C^{pq}_{jk}h_{pq}
\]
generates of an analytic semigroup $e^{t\ope L}$ on $h^{0,\alpha}(\sym\mfd )$.
\end{theorem}

We provide a proof in \S~\ref{sec:L generates}. In the remainder of this section, we recall some well-known consequences of
Theorem~\ref{lem:an-sgp}.  
In particular, we recall the short-time existence and analytic dependence on initial data for nonlinear parabolic equations that can be proved using analytic semigroup methods. Note that the arguments in the remainder of this section work even if $(\mc M,g)$ is not Ricci flat, but merely in $h^{2,\alpha}$. 
While this result suffices for our
purposes, we should point out that Miles Simon, using different methods, has significantly relaxed the required initial regularity.
See the recent survey~\cite{MR4862085} and the references therein.  \medskip

For any background metric $g$, the right-hand side of \eqref{eq:DeT} defines a nonlinear map $\ope F$ on the open subset of
$h\in h^{2,\alpha}(\sym\mfd)$ for which $g$ and $g+h$ define equivalent metrics, \emph{i.e.,}
\[
\mc O^{2,\alpha}_g = \bigl\{ h\in h^{2,\alpha}(\sym\mfd) \mid \exists \epsilon>0 : \epsilon g \leq g+h \leq \tfrac{1}{\epsilon}g \bigr\}.
\]
The map
\[
\ope F : \mc O^{2,\alpha}_g \subset h^{2,\alpha}(\sym\mfd )\rightarrow h^{0,\alpha}(\sym\mfd )
\]
is given by
\[
\big(\ope F(h)\big)_{jk} \stackrel{\rm def}= G^{\ell m}(h)\D_\ell \D_m h_{jk}
+ 2R_{\ell jk}^p h_p^\ell  - R_j^ph_{kp}  - R_k^p h_{jp}\\
-2R_{jk} + Q_{jk}(h, \D h).
\]
The map $\ope F$ depends on the chosen background metric, but we will not include this in our notation. 

The map $\ope F$ is real analytic, and its Fréchet derivative at $h\in \mc O^{2,\alpha}_g$ is the linear operator
\begin{align*}
  \mr d\ope F_h\cdot\, \eta
  =&  G^{\ell m}(h)\D_\ell \D_m \eta_{jk}    + \frac{\partial Q_{jk}}{\partial \D h}\cdot \D \eta\\
   &\quad + 2\Rm*\eta  - 2\Rc* \eta
     + \Bigl\{(\D_\ell \D_m h)\frac{\partial G^{\ell m}}{\partial h} + \frac{\partial Q_{jk}}{\partial h}\Bigr\}\cdot \eta ,
\end{align*}
which is exactly of the type $\ope L$ from Theorem~\ref{lem:an-sgp}. Furthermore, when $h=0$, the operator $\mr d\ope F_0$ is precisely the
Lichnerowicz Laplacian~\eqref{Lich-Lap}.  \medskip

For every $h_0\in \mc O^{2,\alpha}_g$, there exists a time $T_0\in(0, \infty]$ for which the initial value problem of~\eqref{eq:DeT}
has a unique classical solution,
\[
h_{jk} \in C^0\Bigl([0, T_0); \mc O^{2,\alpha}_g \Bigr) \cap C^1\Bigl([0, T_0); h^{0,\alpha}(\sym\mfd ) \Bigr),
\]
for which $h_{jk}(0, x) = h_{0, jk}(x)$.

The existence time $T_0$ is a lower semicontinuous function of $h_0\in \mc O^{2,\alpha}_g$, and the domain of the semiflow,
\[
\mc D = \Bigl\{(t, h_0) \in [0,\infty)\times\mc O^{2,\alpha}_g \mid 0\leq t < T_0(h_0) \Bigr\},
\]
is open in $[0, \infty)\times \mc O^{2,\alpha}_g$.

The time-$t$ map $h_0\mapsto \phi^t(h_0) = h_{0, jk}(t, \cdot)$ is real analytic.

For justification of these claims, see~\cite{MR1059647}.

\subsection{Linearization near a fixed point}
If our background metric $g$ is Ricci-flat then $h=0$ is a fixed point for the DeTurck flow, \textit{i.e.,} $\phi^t(0)=0$ for all
$t\geq 0$.  By definition of the Fréchet derivative, Ricci-DeTurck flows $\gh(t) = g+h(t) $ near the Ricci-flat metric $g$ are
determined by their initial values $\gh(0)= g+h(0)$ via
\[
\gh(t)	= g + \phi^t(h(0))
		= g + \mr d\phi^t(h_0)) \cdot h(0) + o(\|h(0)\|).
\]

By the chain rule, we have
\[
  \mr d\phi^{t+s}(0) = \mr d\phi^t(0)\cdot \mr d\phi^s(0)
\]
for all $t, s\geq 0$, and hence the operators $\{\mr d\phi^t(0)\}$ form semigroup.  This semigroup is analytic, and its generator is
$\mr d\ope F_0$, as one sees by differentiating the equation $\partial_t h = \ope F[h]$ with respect to the initial value $h(0)$.  In
the sense of semigroups, we therefore have\footnote{ For Fr\'echet differentiable map $\mc H$, we interchangeably
use either notation $\mr d\mc H_0 =\mr d\mc H(0)$.}
\[
\mr d\phi^t(0) = \exp\bigl(t\mr d\ope F_0\bigr).
\]
Since $\mr d\ope F_0$ generates an analytic semigroup, the spectra of $\mr d\phi^t(0)$ and $\mr d\ope F_0$ are related.  The following
result is particularly relevant to our situation:
\begin{lemma}
Let $E$ be a Banach space, and let $A:D(A)\to E$ generate an analytic semigroup $e^{tA}$.  If for some $\alpha\in\R$, the spectrum of
$A$ is disjoint from the line $\{z\in\C : \Re z = \alpha\}$, then the spectrum of $e^{tA}$ is disjoint from the circle
$\{z\in\C : |z| = e^{t\alpha}\}$.
\end{lemma}

\begin{proof}
We outline the proof of this classical statement about (analytic) semigroups.  The spectral theory and functional calculus of unbounded
operators can be found in Chapters VII and VIII of Dunford and Schwartz's book \cite{MR1009162} --- see also\cite{MR915552}.

Since $\sigma(A)$ is a closed subset of the complex plane that is contained in a sector $\{z\in\C : |\arg(a-z)|<\vartheta\}$ for some
$a>0, \vartheta<\frac{\pi}{2}$, there is an $\epsilon>0$ such that $\sigma(A)$ is disjoint from the strip
$\{z\in\C: \alpha-\epsilon\leq \Re z\leq \alpha+\epsilon\}$.  It follows that $\sigma(A)$ decomposes in two closed parts,
$\Sigma_- = \{z\in \sigma(A): \Re z <\alpha-\epsilon\}$, $\Sigma_+=\{z\in\sigma(A) : \Re z>\alpha+\epsilon\}$, and that there is a
corresponding splitting $E= E_-\oplus E_+$ into invariant subspaces $E_\pm$ for $A$, with $\sigma(A|E_\pm) = \Sigma_\pm$.  The
splitting is also invariant under $e^{tA}$.  It follows from $\sigma(A_-)\subset \{z:\Re z<\alpha-\epsilon\}$ that
$\sigma(e^{tA_-})\subset \{z: |z|\leq e^{(\alpha-\epsilon)t}\}$, while $\sigma(A_+)\subset\{z:\Re z > \alpha+\epsilon\}$ implies
$\sigma(e^{tA_-})\subset\{z: |z|\geq e^{(\alpha+\epsilon)t}\}$.
\end{proof}

In particular, if $A$ is the Lichnerowicz Laplacian $\Delta_\ell$, then we will show that
$\sigma(\Delta_\ell) = (-\infty, 0]\cup\{\lambda_j:j\in\N\}$ for an eigenvalue sequence $\lambda_1\geq\lambda_2\geq\lambda_3\geq\cdots\to 0$.
It then follows that \(\sigma\bigl(e^{t\Delta_\ell}\bigr)\) is disjoint from any circle $\{z\in\C : |z|=r\}$ with $r>1$ and
$r\neq e^{t\lambda_j}$  for all $j$.

\subsection{Invariant manifolds, following Irwin}
\label{sec:Irwin Chaperon}

There is a long history of invariant manifold theorems which establish their existence and smoothness under various hypotheses.
Inspired by Joel Robbin's \cite{MR227583} existence proof for \textsc{ode} using the Implicit Function Theorem on Banach spaces,
Irwin~\cite{IrwinProofPseudoStableMfdThm} presented a construction of stable and unstable manifolds for maps of finite-dimensional
spaces based on the Implicit Function Theorem.  Building on Irwin's work, Rafael de~la~Llave and Gene Wayne \cite{MR1337223}, and
later Marc~Chaperon~\cite{Chaperon2002,Chaperon2004}, proved an unstable manifold theorem in which the Banach space may be
infinite-dimensional, and in which the map $f$ need not be a local diffeomorphism.  This last feature is especially important for
our setting, because it allows us to apply the result to the time-$t$ map $\phi^t$ defined by the Ricci--DeTurck flow.

Consider a Banach space $E$, an open subset $\mc U\subset E$, and a Fréchet differentiable map $f:\mc U\to E$ with a fixed point
$p\in\mc U$.  The Fréchet derivative $\mr df_p$ is a bounded linear map on $E$.  We assume that for some $a>0$, the spectrum of
$\mr df_p$ is disjoint from the circle $\{\lambda\in\C : |\lambda|=a\}$, so that the spectrum $\sigma(\mr df_p)$ is the union of two
compact sets
\begin{equation} \label{eq:spectrum decomposition} 
\begin{aligned}
&\sigma(\mr df_p) = \Sigma_a^s \cup \Sigma_a^u, \\
&\Sigma_a^s = \{\lambda\in\sigma(\mr
df_p) : |\lambda| < a\},\quad \Sigma_a^u = \{\lambda\in\sigma(\mr df_p) : |\lambda| > a\}.
\end{aligned}
\end{equation}
Then the Banach space $E$ is the direct sum of two $\mr df_p$-invariant subspaces $E^u$ and $E^s$ such that $\mr df_p|_{E^s}$
expands vectors by at most $a$, and $\mr df_p|_{E^u}$ expands vectors by at least $a$.  More precisely, it follows
from~\eqref{eq:spectrum decomposition} that there is a constant $C$ such that
\begin{equation}\label{eq:stable and unstable expansion rates}
\forall v\in E^s: \|(\mr df_p)^k\cdot v\|\leq Ca^k\|v\|
\quad\mbox{ and }\quad
\forall v\in E^u:
\|(\mr df_p)^k\cdot v\| \geq \frac{a^k}{C}\|v\|.
\end{equation}
If $a>1$, then we define the \emph{local unstable set} $W^u_p(a, \mc U)$ to consist of all points $q\in\mc U$ for which there exists
an ancient orbit\footnote{~In the context of maps, an \emph{ancient orbit} is a sequence~$x_i$ that is defined for all nonpositive
  integers and that satisfies $x_i=f(x_{i-1})$ for all $i\leq 0$. } $\{x_i\mid i\leq 0\}$ of the map $f$ that is contained in
$\mc U$ and that converges geometrically to $p$ as $i\to-\infty$, with $d(x_i, p)\leq Ca^k$ for some $C>0$.

The local unstable set $W^u_p(a, \mc U)$ is invariant under $f$ near $p$ in the sense that, if we set $\mc U_0 = f^{-1}(\mc U)$,
then the definition of $W^u_p$ implies directly that $f\bigl(\mc U_0 \cap W^u_p(a, \mc U)\bigr) \subset W^u_p(a, \mc U)$.

\begin{theorem}		\label{UnstableManifoldTheorem}
Let $f:\mc U\to E$ be as above.  If $a>1$ and if the neighborhood $\mc U$ is sufficiently small, then the local unstable set $W^u_p(a, \mc U)$
is a smooth submanifold of $\mc U$ that contains $p$, and for which $T_pW^u_p(a, \mc U) = E^u$. 

The local unstable manifold $W^u_p(a, \mc U)$ has the same differentiability as the map $f$,
\emph{i.e.,} if $f$ is $C^r$ for $1\leq r\leq \infty$ or $r=\omega$, then $W^u_p(a,\mc U)$ also is $C^r$.
\end{theorem}

A proof of this theorem can be found in the papers~\cite{MR1337223,Chaperon2002,Chaperon2004}.  A noteworthy feature of these
references is the extent to which their methods generalize to new situations.  For example, we found it useful to employ very
similar arguments in~\cite{GAFA}.

Since these citations are written from a highly abstract perspective, Appendix~\ref{sec:invariant manifold construction} provides
a short outline of the proof of Theorem~\ref{UnstableManifoldTheorem} adapted to the situation studied in this paper.

\section{Generation of an analytic semigroup} \label{sec:L generates}

This section is dedicated to the proof of Theorem~\ref{lem:an-sgp}.

\subsection{$\ope L$ generates an analytic semigroup}
We must find $\omega>0$ such that for all $f_{jk}\in h^{0,\alpha}(\sym\mfd )$ and all $\lambda$ with $\Re \lambda>\omega$, there is a
unique $h_{jk}\in h^{2,\alpha}(\sym\mfd )$ such that $\lambda h_{jk}-\ope L h_{jk} = f_{jk}$.

We then have to show that $\|h_{jk}\|_{2,\alpha}\leq C\|f_{jk}\|_{0,\alpha}$ for some $C$ that does not depend on $f_{jk}$.  We do this
by approximating the inverse of $\lambda-\ope L$ locally by the inverse of the constant coefficient operator one finds by ``freezing the
coefficients'' of $\ope L$, and then combining the local inverses using a partition of unity.  An abstract version of this common approach
is presented in \cite{Angenent1999}.

\subsection{Eliminating the lower order terms} The operator
\[
\ope K h_{jk} \stackrel{\rm def}=B^{pqr}_{jk}\nabla_r h_{pq} + C^{pq}_{jk}h_{pq}
\]
is bounded from \(h^{1,\alpha}(\sym\mfd)\) to \(h^{0,\alpha}(\sym\mfd)\).  Since \(h^{1,\alpha}\) is an interpolation space between
\(h^{2,\alpha}\) and \(h^{0,\alpha}\) this implies that \(\ope L:h^{2,\alpha}\to h^{0,\alpha}\) generates an analytic semigroup if and
only if \(\ope L-\ope K\) does.  We may therefore assume that the lower order coefficients \(B^{pqr}_{jk}\) and \(C^{pq}_{jk}\) vanish.

\subsection{Local models for \(\ope L\)} 
We continue to assume that \(\mfd \) has bounded geometry. This lets us assume that there exists \(r_0>0\) such that for every point
\(\mf p\in\mfd \), the exponential map \(\exp_p : \ball_{r_0}(0) \hookrightarrow \mfd \) is a diffeomorphism from
\(\ball_{r_0}(0)\subset T_p\mfd \) onto its image \(\ball_{r_0}(\mf p)\subset\mfd \).  Choosing a basis for \(T_{\mf p}\mfd \), we then get
coordinates \(x^1, \dots, x^n\) on \(\ball_{r_0}(\mf p)\).

Using the pushforward $(\exp_{\mf p})_*$ and pullback $(\exp_{\mf p})^*$ maps associated with
$\exp_{\mf p}: \ball_r(0)\subset\R^n \to \ball_r(\mf p)\subset\mfd $, we can represent the operator $\ope L$ locally by a differential
operator,
\[
\ope L_{\mf p} = (\exp_{\mf p})^* \circ \ope L \circ (\exp_{\mf p})_*
\]
which acts on symmetric 2-tensors in $\ball_r(0)\subset \R^n$ via
\begin{align*}
(\ope L_{\mf p} h)_{jk} &= A^{pq}\D_p \D_q h_{jk}  \\
&= A^{pq}(x)\frac{\partial^2h_{jk}}{\partial x^p\partial x^q}(x) 
+ \tilde B^{pqr}_{jk}(x) \frac{\partial h_{pq}}{\partial x^r}(x) + \tilde C^{pq}_{jk}(x)h_{pq}(x).
\end{align*}
From here on, whenrever we consider tensors on $\ball_{r_0}(\mf p)$, we will represent them in geodesic coordinates and not explicitly
mention the pullback and pushforward operations, $(\exp_{\mf p})^*$ and $(\exp_{\mf p})_*$ respectively, in our computations.

The $\tilde B$ and $\tilde C$ coefficients contain Christoffel symbols and their first derivatives that come from expanding the covariant
derivatives $\D_{(\cdots)}$. Ignoring indices and integer coefficients, we have
\[
\tilde B = \Gamma, \qquad \tilde C = \partial \Gamma + \Gamma + \Gamma*\Gamma.
\]
It follows that $A^{pq}, \tilde B^{pqr}_{jk}, \tilde C^{pq}_{jk}$ are uniformly bounded in $C^1(\ball_r(0))$, and that $A^{pq}(x)$ is
uniformly elliptic, \emph{i.e.,} there exists $\sigma>0$ such that for all $a\in\N$, $x\in\ball_r(0)$, and $\xi\in\R^n$, one has
\begin{equation}
\label{eq:uniformly elliptic}
\sigma |\xi|^2 \leq A^{pq}(x)\xi_p\xi_q \leq \frac{1}{\sigma}|\xi|^2.
\end{equation}

\subsection{The local resolvent}
Let $r_1\in(0,r_0)$ be given, and consider the extended coefficients $\hat A^{pq}, \hat B^{pqr}_{jk}, \hat C^{pq}_{jk}:\R^n\to\R$ obtained
by setting
\[
\hat A^{pq}(x) = A^{pq}(x), \qquad \hat B^{pqr}_{jk}(x) = \tilde B^{pqr}_{jk}(x),\qquad \hat C^{pq}_{jk}(x) = \tilde C^{pq}_{jk}(x) ,
\]
for $|x|\leq r_1$, and by extending $\hat A^{pq}(x)$, $\hat B^{pqr}_{jk}(x)$, and $\hat C^{pq}_{jk}(x)$ so they are constant along rays
emanating from the origin, \emph{i.e.,} for $|x|\geq r_1$, we have
\begin{align*}
\hat A^{pq}(x) &= A^{pq}\Bigl(\frac{r_1x}{|x|}\Bigr), &
                              \hat B^{pqr}_{jk}(x) &=  \tilde B^{pqr}_{jk}\Bigl(\frac{r_1x}{|x|}\Bigr), &
                                                                                                          \hat C^{pq}_{jk}(x) &=  \tilde C^{pq}_{jk}\Bigl(\frac{r_1x}{|x|}\Bigr).
\end{align*}
We then define an elliptic operator on symmetric $2$-tensor fields on $\R^n$ by
\begin{align*}
(\hat {\ope L}_{\mf p} h)_{jk} 
= \hat A^{pq}(x)\frac{\partial^2h_{jk}}{\partial x^p\partial x^q}(x) 
+ \hat B^{pqr}_{jk}(x) \frac{\partial h_{pq}}{\partial x^r}(x) + \hat C^{pq}_{jk}(x)h_{pq}(x).
\end{align*}

\begin{lemma} \label{lem:local resolvent} There exist $r_1\in(0, r_0)$, $\omega>0$, and $C>0$ such that
\[
\lambda-\hat{\ope L} _{\mf p} : h^{2, \alpha}(\sym\R^n)\to h^{0,\alpha}(\sym\R^n)
\]
is invertible for all $\lambda\in\C$ with $\Re\lambda > \omega$, and such that
\[
\|(\lambda-\hat{\ope L} _{\mf p})^{-1}f_{jk}\|_{2,\alpha} \leq C \|f_{jk}\|_{0,\alpha}
\]
for all $f_{jk}\in h^{0,\alpha}(\sym\R^n)$.

The constants $C, \omega, r_1$ do not depend on the chosen point $\mf p\in\mfd $.
\end{lemma}

As usual, the identity
$\lambda (\lambda-\hat{\ope L} _{\mf p})^{-1} = \boldsymbol{1} + \hat{\ope L} _{\mf p}\circ (\lambda-\hat{\ope L} _{\mf p})^{-1}$ then
implies that
\begin{equation}
\label{eq:resolvent decay}
\|(\lambda-\hat{\ope L} _{\mf p})^{-1}f_{jk}\|_{0, \alpha} \leq \frac{C}{|\lambda|}\|f_{jk}\|_{0,\alpha},\qquad
\forall f_{jk}\in h^{0,\alpha},\; \Re \lambda \geq \omega. 
\end{equation}
By interpolation between the $h^{0,\alpha}$ and $h^{2,\alpha}$ norms, one also obtains the estimate
\begin{equation}\label{eq:resolvent interpolated}
\|(\lambda-\hat{\ope L} _{\mf p})^{-1}f_{jk}\|_{1, \alpha} \leq \frac{C}{\sqrt{|\lambda|}}\|f_{jk}\|_{0,\alpha},\qquad
\forall f_{jk}\in h^{0,\alpha},\;\Re \lambda \geq \omega. 
\end{equation}
The constant $C$ is the same for all $\mf p\in\mfd $.

\begin{proof}
The gradients of $A^{pq}, \tilde B^{pqr}_{jk}, \tilde C^{pq}_{jk}$ are uniformly bounded by assumption.  It follows that the extensions
$\hat A^{pq}, \hat B^{pqr}_{jk}, \hat C^{pq}_{jk}$ also have uniformly bounded gradients on all of $\R^n$.  Let $M$ be an upper bound for
those gradients.  Then we have
\begin{equation}
\label{eq:coeffs C0 close}
|\hat A^{pq}(x)- A^{pq}(0)| +
|\hat B^{pqr}_{jk}(x)- \tilde B^{pqr}_{jk}(0)| 
+   |\hat C^{pq}_{jk}(x)- \tilde C^{pq}_{jk}(0)|
\leq Mr_1
\end{equation}
for all $x$ with $|x|\leq r_1$.  Since the coefficients are constant along rays outside of the ball $\ball_{r_1}(0)$, it follows that
\eqref{eq:coeffs C0 close} also holds for all other $x\in\R^n$.

Because the gradients of $\hat A^{pq}, \hat B^{pqr}_{jk}, \hat C^{pq}_{jk}$ are all bounded by $M$, we then find by interpolation that
\begin{multline}
\label{eq:coeffs Ca close}
\|\hat A^{pq}(\cdot)- A^{pq}(0)\|_{0, \alpha} +
\|\hat B^{pqr}_{jk}(\cdot)- \tilde B^{pqr}_{jk}(0)\|_{0, \alpha}  \\
+ \|\hat C^{pq}_{jk}(\cdot)- \tilde C^{pq}_{jk}(0)\|_{0, \alpha} \leq Cr_1^{1-\alpha}.
\end{multline}
Hence $\ope L_{\mf p}$ is close to the constant coefficient operator
\[
\ope M h_{jk} \stackrel{\rm def}= \hat A^{pq}(0)\frac{\partial^2h_{jk}}{\partial x^p\partial x^q}(x) + \hat B^{pqr}_{jk}(0) \frac{\partial
h_{pq}}{\partial x^r}(x) + \hat C^{pq}_{jk}(0)h_{pq}(x)
\]
in the sense that for all $f_{jk}\in h^{2, \alpha}(\sym\R^n)$, one has
\[
\|\hat{\ope L} _{\mf p}f_{jk} - \ope M f_{jk}\|_{h^{0,\alpha}} \leq Cr_1^{1-\alpha} \|f_{jk}\|_{h^{2, \alpha}(\sym\R^n)}.
\]
The coefficients of $\ope M$ are uniformly bounded, and uniformly elliptic, so the resolvent $(\lambda-\ope M)^{-1}$ is defined for all
$\lambda\in\C$ with $\Re\lambda>\omega$ for some $\omega>0$, and the $h^{0,\alpha}\to h^{2,\alpha}$ operator norms of
$(\lambda-\ope M)^{-1}$ are uniformly bounded by some $C>0$.

Therefore, if $r_1>0$ is small enough, then the geometric series
\begin{align*}
(\lambda-\hat{\ope L} _{\mf p})^{-1}
&=  (\lambda-\ope M + \ope M -\hat{\ope L} _{\mf p})^{-1}  \\
&=  (\lambda-\ope M)^{-1}(\boldsymbol{1} + (\ope M -\hat{\ope L} _{\mf p})(\lambda-\ope M)^{-1})^{-1}  \\
&=   (\lambda-\ope M)^{-1} + \sum_{k\geq 1}\Bigl(
-(\ope M -\hat{\ope L} _{\mf p}) (\lambda-\ope M)^{-1}
\Bigr)^k
\end{align*}
converges, so that $(\lambda-\hat{\ope L} _{\mf p})^{-1}$ exists for all $\lambda\in\C$ with $\Re\lambda>\omega$ and is bounded by $2C$.
\end{proof}

\subsection{The resolvent of $\ope L$}
Since $\mfd $ has bounded geometry, there exists $N\in \N$ such that for any $r\in(0, r_1)$ there exists a sequence of points
$\{\mf p_a\in\mfd \mid a\in\N\}$ for which $\{\ball_a = \ball_r(\mf p_a) \mid a\in\N\}$ is a covering of $\mfd $ with the following
property:
\begin{equation}
\label{eq-locally finite cover}
\parbox[c]{260pt}{\raggedright\itshape
No more than $N$ distinct balls from the covering $\{\ball_a\}$ have nonempty intersection, \emph{i.e.,} one has $\cap_{k=0}^N \ball_{a_k} = \varnothing$
for any choice of indices $a_0<a_1<\cdots<a_N\in\N$.}
\end{equation}  

For a proof, which only needs completeness and a lower bound for $\Rc$, see Lemma~1.6 of~\cite{Hebey}.  \medskip

For each $a\in\N$ let $\hat{\ope L} _{\mf p_a}$ be the localized operator at the point $\mf p_a$, and let
\[
\resolvent_a(\lambda) = (\lambda -\hat{\ope L} _{\mf p_a})^{-1} : h^{0,\alpha}(\sym\R^n) \to h^{{2,\alpha}}(\sym\R^n)
\]
be its resolvent, as constructed in Lemma~\ref{lem:local resolvent}.

To solve $\lambda f_{jk} - \ope L f_{jk} = h_{jk}$ on $\mfd$, we localize $h_{jk}$ to each geodesic ball $\ball_a$, apply the local
resolvent $\resolvent_a(\lambda)$, and then glue the results together, summing over $a\in\N$ to get an approximation to
$(\lambda-\ope L)^{-1}h_{jk}$.  In the localization and subsequent gluing, we will use a partition of unity $\varphi_a\in C^\infty(\mfd )$
in the sense of \S~\ref{sec:partition of unity}, \emph{i.e.,} $\mathop{\mathrm{supp}}\varphi_a\subset\ball_r(p_a)$, and
$\sum_a \varphi_a(x)^2=1$.

\subsubsection{The approximate resolvent} Consider the operator
\begin{equation}
\label{eq:approx-resolvent}
\ope S(\lambda)h_{jk} \stackrel{\rm def}{=}
\sum_{a\in\N} \varphi_a(x) \resolvent_a(\lambda)\bigl(\varphi_a \cdot h_{jk}\bigr).
\end{equation}
The terms in this sum should be interpreted as follows.  The tensor field $h_{jk}$ is defined on all of $\mfd $.  For any given
$a\in\N$, we multiply it with $\varphi_a$, which leads to a tensor field $\varphi_a h_{jk}$ whose support lies in $\ball_a$ and which in
geodesic coordinates defines a compactly supported tensor field on $\ball_{r}(0)\subset\R^n$.  We may therefore think of $\varphi_a h_{jk}$
as a tensor field on $\R^n$ that is supported in $\ball_r(0)$, and to which one can apply the local resolvent $\resolvent_a(\lambda)$,
resulting in $\resolvent_a(\varphi_a h_{jk})$.  The tensor field $\resolvent_a(\varphi_a h_{jk})$ is defined on all of $\R^n$.  To get a
tensor field on $\mfd $ again, we restrict $\resolvent_a(\lambda)(\varphi_a h_{jk})$ to $\ball_r(0)\subset\R^n$.  Identifying $\ball_r(0)$
with $\ball_a = \ball_r(p_a)\subset\mfd $ via the exponential map, we get a tensor field defined on $\ball_a$.  Multiplying it with
$\varphi_a$ we get a compactly supported tensor field in $\ball_a$, which we can extend to the whole manifold by letting it vanish on
$\mfd \setminus\ball_a$.  In other words, in \eqref{eq:approx-resolvent},
$\varphi_a(x) \resolvent_a(\lambda)\bigl(\varphi_a \cdot h_{jk}\bigr)$ is an abbreviation for
\[
\varphi_a(x) (\exp_{p_a})_*\biggl\{ \resolvent_a(\lambda)\Bigl[(\exp_{p_a})^*\bigl(\varphi_a \cdot h_{jk}\bigr)\Bigr]\biggr\}.
\]

\subsubsection{Boundedness of $\ope S(\lambda)$} \label{sec:S lambda bounded}

We verify that $\ope S(\lambda):h^{0,\alpha}(\sym\mfd )\to h^{2,\alpha}(\sym\mfd )$ is uniformly bounded, \emph{i.e.,} there exists $C>0$
such that for all $h_{jk}\in h^{0,\alpha}(\sym\mfd )$ and $\lambda\in\C$ with $\Re\lambda > \omega$, one has
\[
\|\ope S(\lambda)h_{jk}\|_{h^{2, \alpha}(\sym\mfd )} \leq C \|h_{jk}\|_{h^{0, \alpha}(\sym\mfd )} .
\]
\begin{proof}
The $h^{2, \alpha}$ norm is equivalent to
\[ [h_{jk}]_{2,\alpha} \stackrel{\rm def}{=} \sup_{a\in\N} \bigl\|h_{jk}\big|_{\ball_a}\bigr\|_{h^{2, \alpha}}.
\]
At this point, we recall from~\eqref{eq-locally finite cover} that the intersection of more than $N$ balls in the covering
$\{\ball_a\}_{a\in\N}$ is empty.  This implies that in the sum that defines $\ope S(\lambda)h_{jk}(x)$, the only nonzero terms are those for
which $x\in\ball_a$.  For any $x\in\mfd $, there are at most $N$ such terms. We therefore get
\begin{align*}
\|\ope S(\lambda)h_{jk}\|_{2, \alpha}
&\leq C \sup_{a\in\N} \|(\ope S(\lambda)h_{jk})\big|_{\ball_a}\|_{2, \alpha}\\
&\leq C \sup_{a\in\N}\sum_{b\in\N} 
\bigl\|\bigl(\varphi_b\resolvent_b(\lambda)(\varphi_b h_{jk})\bigr)\Big|_{\ball_a}\bigr\|_{2,\alpha} \\
&\leq C \sup_{a\in\N}\sum_{\ball_a\cap\ball_b\neq\varnothing} 
\bigl\|\bigl(\varphi_b\resolvent_b(\lambda)(\varphi_b h_{jk})\bigr)\Big|_{\ball_a}\bigr\|_{2,\alpha} \\
&\leq C \sup_{a\in\N}\sum_{\ball_a\cap\ball_b\neq\varnothing} 
\bigl\|\varphi_b\resolvent_b(\lambda)(\varphi_b h_{jk})\bigr\|_{2,\alpha} \\
&\leq NC \sup_{b\in\N} 
\bigl\|\varphi_b\resolvent_b(\lambda)(\varphi_b h_{jk})\bigr\|_{2,\alpha} \\
&\leq NC \sup_{b\in\N} 
\bigl\|\resolvent_b(\lambda)(\varphi_b h_{jk})\bigr\|_{2,\alpha} \\
&\leq NC \sup_{b\in\N} 
\bigl\|\varphi_b h_{jk}\bigr\|_{0,\alpha} \\
&\leq NC \|h_{jk}\|_{0,\alpha}.
\end{align*}
\end{proof}

\subsubsection{The error in the approximate resolvent}
We proceed to show that for any $f_{jk}\in h^{0,\alpha}(\sym\mfd )$, one has
\begin{equation}
\label{eq:Error in approx resolvent}
(\lambda-\ope L)\ope S(\lambda)f_{jk} 
= f_{jk} -\ope E(f_{jk})
= \bigl(\boldsymbol{1}_{h^{0,\alpha}(\mfd )} -\ope E\bigr)(f_{jk}),
\end{equation}
where $\ope E$ is the operator
\[
\ope E(f_{jk}) = \sum_a [\ope L,\varphi_a] \resolvent_a(\lambda) (\varphi_a \cdot f_{jk}).
\]
\begin{proof} Let $f_{jk}\in h^{0,\alpha}(\sym\mfd )$ be given.  Then
\begin{align*}
(\lambda-\ope L)\ope S(\lambda)f_{jk}
&=(\lambda-\ope L)\sum_a \varphi_a \resolvent_a(\lambda) (\varphi_a \cdot f_{jk}) \\
&=\sum_a (\lambda-\ope L)\varphi_a \resolvent_a(\lambda) (\varphi_a \cdot f_{jk}) \\
&=\sum_a \varphi_a (\lambda-\ope L)\resolvent_a(\lambda) (\varphi_a \cdot f_{jk}) 
- [\ope L,\varphi_a] \resolvent_a(\lambda) (\varphi_a \cdot f_{jk}).
\end{align*}
On the support of $\varphi_a$, the operator $\ope L$ coincides with our extended operator $\hat{\ope L} _{\mf p_a}$, so that
$\varphi_a(\lambda-\ope L)\resolvent_a(\lambda)=\varphi_a$.  Hence,
\begin{align*}
(\lambda-\ope L)\ope S(\lambda)f_{jk}
&=\sum_a \varphi_a \cdot(\varphi_a \cdot f_{jk}) 
- [\ope L,\varphi_a] \resolvent_a(\lambda) (\varphi_a \cdot f_{jk}) \\
&=f_{jk}
- \sum_{a} [\ope L,\varphi_a] \resolvent_a(\lambda) (\varphi_a \cdot f_{jk}).
\end{align*}
\end{proof}

\subsubsection{Estimate for the error term}
For all $\lambda\in\C$ with $\Re\lambda>\omega$ and all tensor fields $h_{jk}\in h^{0,\alpha}(\sym\mfd )$, we will show that
\begin{equation} \label{eq:Error bounded} \|\ope E h_{jk}\|_{0,\alpha} \leq \frac{C}{\sqrt{|\lambda|}} \|h_{jk}\|_{0,\alpha}.
\end{equation}

\begin{proof}
The key observation here is that the commutator $[\ope L, \varphi_a]$ is a first-order differential operator: in geodesic coordinates, we
have
\begin{align*} [\ope L, \varphi_a]f_{jk} &=
               A^{pq}_{(a)}(x) [\partial_p\partial_q, \varphi_a]f_{jk} + \tilde B^{pqr}_{jk(a)} [\partial_r, \varphi_a]f_{pq}\\
             &=A^{pq}_{(a)}\frac{\partial^2\varphi_a}{\partial x^p\partial x^q}f_{jk} + \tilde B^{pqr}_{jk(a)}
               \frac{\partial\varphi_a}{\partial x^r} f_{pq} + 2A^{pq}_{(a)} \frac{\partial\varphi_a}{\partial
               x^p}\frac{\partial f_{jk}}{\partial x^q}.
\end{align*}
This implies that
\[
\|[\ope L, \varphi_a]f_{jk}\|_{0,\alpha} \leq C \|f_{jk}\|_{1, \alpha}
\]
and that
\[
\|[\ope L,\varphi_a]\resolvent_a(\lambda)\varphi_a h_{jk}\|_{0,\alpha} \leq C \|\resolvent_a(\lambda)\varphi_a h_{jk}\|_{1,\alpha}.
\]
By the interpolation estimate~\eqref{eq:resolvent interpolated}, one then has
\[
\|[\ope L,\varphi_a]\resolvent_a(\lambda)\varphi_a h_{jk}\|_{0,\alpha} \leq \frac{C}{\sqrt{|\lambda|}} \|\varphi_a h_{jk}\|_{0,\alpha} \leq
\frac{C}{\sqrt{|\lambda|}} \|h_{jk}\|_{0,\alpha}.
\]
The same reasoning as in \S\ref{sec:S lambda bounded} then leads to the bound~\eqref{eq:Error bounded} for $\ope E$.
\end{proof}

\subsubsection{Construction of the resolvent}
We showed in~\eqref{eq:Error bounded} that $\|\ope E\|_{0,\alpha\to0,\alpha} \leq \frac 12$ if $|\lambda|$ is large enough.  It follows that
for large $|\lambda|$, the operator $\boldsymbol{1}-\ope E$ on $h^{0, \alpha}(\sym \mfd )$ has a bounded inverse given by the geometric
series
\[
\bigl(\boldsymbol{1}-\ope E\bigr)^{-1} = \boldsymbol{1} + \ope E + \ope E^2 + \ope E^3 + \cdots.
\]
Together with~\eqref{eq:Error in approx resolvent}, this implies that
\[
(\lambda-\ope L) \ope S(\lambda)(\boldsymbol{1}-\ope E)^{-1} = \boldsymbol{1},
\]
\emph{i.e.,} that $\ope S(\lambda)(\boldsymbol{1}-\ope E)^{-1}$ is a right inverse for $\lambda-\ope L$.

We have obtained this right inverse for all $\lambda$ with sufficiently large $|\lambda|$, and thus there exists $\omega'>0$ such that
$\ope S(\lambda)(\boldsymbol{1}-\ope E)^{-1}$ is a right inverse whenever $\Re \lambda>\omega'$. Moreover, the norm of the right inverse is
bounded by
\begin{multline} \label{eq:resolvent decay ok} \|\ope S(\lambda)(\boldsymbol{1}-\ope E)^{-1}\|_{0,\alpha\to0,\alpha}
\leq \\
\|\ope S(\lambda)\|_{0,\alpha\to0,\alpha} \|(\boldsymbol{1}-\ope E)^{-1}\|_{0,\alpha\to0,\alpha} \leq 2 \|\ope
S(\lambda)\|_{0,\alpha\to0,\alpha} \leq \frac{C}{|\lambda|}.
\end{multline}

\subsubsection{Left invertibility} \label{sec:left inverse}
We conclude by showing that the right inverse is also a left inverse,
\emph{i.e.,} by showing that $\lambda-\ope L:h^{2,\alpha}\to h^{0,\alpha}$ is injective if $|\lambda|$ is sufficiently large.

Suppose that $f_{jk}\in h^{2,\alpha}(\sym\mfd )$ satisfies $(\lambda-\ope L)f_{jk}=0$.  Then we have
\[
\ope S(\lambda)(\lambda-\ope L)f_{jk}=0,
\]
and hence
\[
\sum_{a}\varphi_a \resolvent_a(\lambda) \varphi_a (\lambda-\ope L)f_{jk}=0.
\]
Arguing as before, we commute $\varphi_a$ and $\ope L$ to obtain
\[
\sum_{a}\varphi_a \resolvent_a(\lambda)(\lambda-\ope L) \varphi_a f_{jk} +\varphi_a \resolvent_a(\lambda) [\ope L, \varphi_a]f_{jk}=0.
\]
Using the operator identity $\varphi_a \resolvent_a(\lambda)(\lambda-\ope L) \varphi_a = \varphi_a^2$ ,we find that
\[
f_{jk}= \sum_{a} \varphi_a \resolvent_a(\lambda) [\ope L, \varphi_a]f_{jk}.
\]
Once again, we note that $[\ope L, \varphi_a]$ is a first-order differential operator with $C^1$ coefficients, so that
$[\ope L, \varphi_a] : h^{1,\alpha} \to h^{0,\alpha}$ is bounded.  Therefore, one has
\begin{align*}
\|f_{jk}\|_{1,\alpha} 
&\leq N\sup_a \|\varphi_a \resolvent_a(\lambda) [\ope L, \varphi_a]f_{jk}\|_{1,\alpha} \\
&\leq \frac{CN}{\sqrt{|\lambda|}}\sup_{a}\|[\ope L, \varphi_a]f_{jk}\|_{0,\alpha} \\
&\leq \frac{CN}{\sqrt{|\lambda|}}\|f_{jk}\|_{1,\alpha}.
\end{align*}
If $|\lambda| > CN$, this implies that $f_{jk}=0$.

We have shown that $\lambda-\ope L$ is injective and has a right inverse for all large enough $|\lambda|$, and that the inverse
$(\lambda-\ope L)^{-1}$ satisfies the required bound \eqref{eq:resolvent decay ok}.  Therefore, $\ope L$ does indeed generate an analytic
semigroup on $h^{0,\alpha}(\sym\mfd )$. This completes our proof of Theorem~\ref{lem:an-sgp}.

\section{The spectrum of $\Delta_\ell$ on asymptotically flat manifolds} \label{Spectrum of LL}

The Lichnerowicz Laplacian is given by~\eqref{Lich-Lap}, \emph{i.e.,}
\[
\Delta_\ell h_{jk} = \Delta_{\mfd }+ 2 R^q_{pjk}h_q^p - R_j^p h_{pk}-R_k^q h_{jq},
\]
where
\[
\Delta_{\mfd } h_{jk} \stackrel{\rm def}{=} g^{pq}\D_p\D_q h_{jk}
\]
is the ``rough Laplacian.''

\begin{lemma} If $|\Rm(\mf p)|\to 0$ as $\mf p\to\infty$ and $r(\mf p)\to\infty$ as $\mf p\to\infty$, which are implied by
Assumption~\ref{Asymptotically flat}, then it is true that:
\begin{enumerate}[\rm(1)]
\item $\Delta_\ell : h^{2,\alpha}\to h^{0,\alpha}$ generates an analytic semigroup,
\item $\sigma(\Delta_{\mfd }) = (-\infty, 0]$, and
\item $\sigma_{\rm ess}(\Delta_\ell) = (-\infty, 0]$.
\end{enumerate}
\end{lemma}

We already proved (1) in Theorem~\ref{lem:an-sgp}.  Statement (3) follows from (2) because $\Delta_\ell$ is a compact perturbation of
$\Delta_{\mfd }$.  We therefore only have to establish (2), which we proceed to do in the remainder of this section.  \medskip

Let $\lambda\in\C\setminus(-\infty,0]$ be given.  We show below that $\lambda-\Delta_\mfd : h^{2,\alpha}\to h^{0, \alpha}$ has a bounded
inverse.  If we were considering $\Delta_\mfd$ as an unbounded operator in $L^2(\mfd; \sym\mfd)$, then self-adjointness of $\Delta_\mfd$
would guarantee that $\lambda-\Delta_\mfd : H^2(\sym\mfd)\to L^2(\sym\mfd)$ is invertible.  However, on noncompact manifolds, the spectrum
of an elliptic operator like $\Delta_\mfd$ may depend on the particular function space in which the resolvent $(\lambda-\Delta_\mfd)^{-1}$
should act, and the spectrum of $\Delta_\mfd$ can be strictly larger than just $(-\infty,0]$.  Our goal in this section is to show that the
asymptotic flatness of $\mfd$ rules out this possibility.

\subsection{The resolvent at $\infty$} \label{ResolventAtInfinity}
Let $\lambda\in\C\setminus(-\infty,0]$.  Then the operator
$\lambda-\Delta_{\R^n}:h^{2,\alpha}(\R^n) \to h^{0,\alpha}(\R^n)$ has a bounded inverse
\[
(\lambda-\Delta_{\R^n})^{-1} : h^{0,\alpha}(\R^n) \to h^{2,\alpha}(\R^n).
\]
Our assumption that $(\mfd , g)$ is asymptotically flat lets us use the resolvent of the Laplacian on $\R^n$ to approximate the resolvent of
$\lambda-\Delta_{\mfd }$ in the regions where the manifold is close to flat.

Let constants $\varepsilon>0$, $R>0$ be given.  Then there is a compact set $K_{R, \varepsilon}\subset\mfd$ such that
$\mr{inj}(\mf p)>2R$ for all $\mf p\in\mfd\setminus K_{R,\varepsilon}$, and such that the components of the metric in geodesic coordinates
$\exp_{\mf p} : \ball_{2R}(0) \to \ball_{2R}(\mf p)$ satisfy
\begin{equation} \label{Lip1} \|g_{pq}-\delta_{pq}\|_{h^{2,\alpha}(\ball_{2R}(0))} \leq \varepsilon.
\end{equation}
If $\varepsilon$ is sufficiently small, then it follows that the inverse $g^{pq}(x)$ also satisfies
\begin{equation} \label{Lip2} \|g^{pq} - \delta_{pq}\|_{h^{2,\alpha}(\ball_{2R}(0))} \leq C\varepsilon.
\end{equation}
In particular, the Christoffel symbols of the metric and their derivatives are also small in the sense that
\[
\|\Gamma^i_{jk}(x)\|_{h^{1,\alpha}(\ball_{2R}(0))} \leq C\varepsilon.
\]
In geodesic coordinates, the rough Laplacian of a tensor field $h_{jk}$ is given by
\[
\Delta_{\mfd } h_{jk} = g^{pq}(x) \partial_p\partial_q h_{jk} + B^{pqr}_{jk}(x)\partial_{r} h_{pq}(x) + C^{pq}_{jk}(x) h_{pq}(x),
\]
where the coefficients $B, C$ are related to the Christoffel symbols via expressions of the form $B = \Gamma$ and
$C = \partial \Gamma + \Gamma*\Gamma$.  This implies that
\begin{equation}\label{eq:nearly flat}
\|\Delta_\mfd h_{jk} - \Delta_{\R^n} h_{jk} \|_{h^{0,\alpha}(\ball_{2R}(0))} \leq C\varepsilon \|h_{jk}\|_{h^{2,\alpha}(\ball_{2R}(0))}
\end{equation}
for all $h_{jk}\in h^{2,\alpha}(\ball_{2R}(0)) $.  \medskip

Our bounded geometry assumption together with a standard covering argument now let us construct a locally finite covering as follows.  We
choose points $\mf p_a$ for $a\in\N$ such that
\begin{enumerate}
\item if $\ball_a^\prime \stackrel{\rm def}=B_{R/2}(\mf p_a)$, then $\{\ball_a^\prime:a\in\mb N\}$ is an open cover of
$\overline{\mfd\setminus K_{R,\varepsilon}}$, and
\item if $\ball_a \stackrel{\rm def}=B_{R}(\mf p_a)$, one has $\ball_{a_0}\cap\cdots\cap \ball_{a_N}=\emptyset$ for all
$a_0<a_1<\cdots<a_N$.
\end{enumerate}

Choose $\eta\in C^\infty(\mb R)$ such that
\[
\eta(s)\left\{\;
\begin{matrix} =1 & \mbox{for }s\leq\frac12,\\
>0 & \mbox{for } s<1,\\
=0 & \mbox{for }s\geq1.\end{matrix}\right.
\]
We define
\[
\psi_a(x) \stackrel{\rm def}= \eta\Big(\frac{d(x,\mf p_a)}{R}\Big) \qquad\mbox{ and }\qquad \vp(x) \stackrel{\rm def}=
\frac{\psi_a(x)}{\sum_b \psi_b(x)},
\]
noting that $\sum_a\vp_a(x)=1$ on $\bigcup_a \ball_a$.  Furthermore, because
$\overline{\mfd\setminus K_{R,\varepsilon}}\subset\bigcap_a \ball_a^\prime$, for any $x\in \overline{\mfd\setminus K_{R,\varepsilon}}$,
there exists $a\in\mb N$ such that $x\in \ball_a^\prime$, and thus we have $\psi_a(x)\geq\eta(\frac12)=1$. Therefore,
$\sum_a\psi_a(x)\geq1$, and $\vp$ is smooth on $\overline{\mfd\setminus K_{R,\varepsilon}}$.

Define
\[
\lambda(x) \stackrel{\rm def}= \sum_{d(p_a,K_{R,\varepsilon})\geq2R}\vp_a(x).
\]

Now if $d(x,K_{R,\varepsilon})\geq3R$, then for any index $a$ such that $x\in \ball_a$, one has $d(\mf p_a,K_{R,\varepsilon})\geq2R$. For
such $x$, it follows that
\[
\sum_b\psi_b(x) = \sum_{\psi_b(x)>0}\psi_b(x) = \sum_{\ball_b\ni x}\psi_b(x) = \sum_{d(\mf p_b,K_{R,\varepsilon})\geq2R}\psi_b(x),
\]
and hence
\[
\lambda(x)\Big|_{d(x,K_{R,\varepsilon})\geq3R} = \sum_{d(\mf p_a,K_{R,\varepsilon})\geq2R}\frac{\psi_a(x)}{\sum_b\psi_b(x)}
=\frac{\sum_{d(\mf p_a,K_{R,\varepsilon})\geq2R}\psi_a(x)}{\sum_b\psi_b(x)}=1.
\]

On the other hand, if $\lambda(x)>0$ then there exists $\ball_a\ni x$ centered at $\mf p_a$ such that $d(\mf p_a,K_{R,\varepsilon})\geq2R$,
which implies that $d(x,K_{R,\varepsilon})\geq R$. Now define
\[
\vp_0(x)\stackrel{\rm def}= 1-\lambda(x).
\]
Then $\vp_0(x)=1$ on $x\in K_{R,\varepsilon}$.  Let
\[
\mf A = \{0\}\cup \{a: d(\mf p_a, K_{R,\varepsilon})\geq 2R\}.
\]
We have shown that $\{\vp_a\colon a\in \mf A\}$ is a partition of unity. Moreover, the following observation is an immediate consequence of
its construction:

\begin{lemma} \label{lem:cutoff grad bounds} There exists a constant $C$ depending only on the cutoff function $\eta$, but not $R$ or $N$,
such that for all $a\in \mf A$, we have
\[
\|\nabla\vp_a\|_{0,\alpha} +\|\nabla^2\vp_a\|_{0,\alpha} \leq \frac{C}{R}.
\]
\end{lemma}
\medskip

The set
\[
\Omega \stackrel{\rm def}= \{x\in\mc M : \vp_0(x)>0\}
\]
is an open neighborhood of $K_{R, \varepsilon}$ with compact closure.  We write $\Delta_\Omega$ for the rough Laplacian restricted to symmetric
tensors on $\Omega$.  Using standard $L^2$ theory, one shows that $ \Delta_\Omega : H_0^1(\sym\Omega) \to H^{-1}(\sym\Omega) $ defines a
self-adjoint operator,
\[
\Delta_\Omega : \mc D(\Delta_\Omega)\cap H_0^1(\sym\Omega)\to L^2(\sym\Omega),
\]
with compact resolvent, whose spectrum consists of point spectrum contained in $(-\infty,0]$.  The resolvent of $\Delta_\Omega$ on $\Omega$
with Dirichlet boundary conditions is therefore a bounded operator
\[
(\lambda - \Delta_\Omega)^{-1} : L^2(\sym\Omega) \to \mc D(\Delta_\Omega) \cap H_0^1(\sym\Omega).
\]
Interior Schauder estimates for elliptic equations imply for any open subset $\Omega_1$ of $\mfd$ with $\bar\Omega_1\subset\Omega$ and any
$h_{jk}\in h^{0,\alpha}(\sym\Omega)$, the restriction of $(\lambda-\Delta_\Omega)^{-1}h_{jk}$ to $\Omega_1$ belongs to $h^{2,\alpha}$.
Furthermore one has the bound
\[
\bigl\|(\lambda-\Delta_\Omega)^{-1}h_{jk}\bigr\|_{h^{2,\alpha}(\Omega_1)} \leq C \|h_{jk}\|_{h^{0,\alpha}(\Omega)},
\]
where the constant $C$ does not depend on $h_{jk}$.

We consider the following approximate resolvent
\[
\ope S = \varphi_0 (\lambda-\Delta_\Omega)^{-1}\varphi_0 + \sum_{a\in \mf A, a > 0} \varphi_a (\lambda-\Delta_{\R^n})^{-1}\varphi_a,
\]
where the terms are again to be interpreted as in~\eqref{eq:approx-resolvent}; thus, we multiply a tensor field with $\varphi_a$, express it
in geodesic coordinates on $\ball_a$, apply the resolvent of $\Delta_{\R^n}$ or $\Delta_\Omega$, and multiply again with $\varphi_a$,
resulting in an $h^{2,\alpha}$ function on $\mfd $; finally we add these over $a\in\mf A$.

From here on $\sum_a$ will always denote a sum over all $a\in\mf A$.
\medskip

As before, we estimate the error in this near-resolvent.  Let $\Delta_a = \Delta_\Omega$ for $a=0$ and $\Delta_a = \Delta_{\R^n}$ for
$a\in\mf A$ with $a>0$.  Then
\begin{align*}
(\lambda-\Delta_{\mfd })\ope S
&= \sum_a (\lambda-\Delta_{\mfd })\varphi_a (\lambda-\Delta_a)^{-1}\varphi_a   \\
&= \sum_a \varphi_a (\lambda-\Delta_{\mfd })(\lambda-\Delta_a)^{-1}\varphi_a  
-\sum_a [\Delta_{\mfd }, \varphi_a](\lambda-\Delta_a)^{-1}\varphi_a \\
&=\boldsymbol{1} +  \sum_{a} \varphi_a (\Delta_a-\Delta_{\mfd })(\lambda-\Delta_a)^{-1}\varphi_a  
-\sum_a [\Delta_{\mfd }, \varphi_a](\lambda-\Delta_{a})^{-1}\varphi_a \\
&=\boldsymbol{1} +  \sum_{a>0} \varphi_a (\Delta_{\R^n}-\Delta_{\mfd })(\lambda-\Delta_a)^{-1}\varphi_a  
-\sum_a [\Delta_{\mfd }, \varphi_a](\lambda-\Delta_{a})^{-1}\varphi_a.
\end{align*}
The term with $a = 0$ does not appear in the first summation, because when $a=0$, we have $\Delta_a = \Delta_\Omega=\Delta_\mfd$. We define
errors $\ope E_1, \ope E_2$ by
\begin{align*}
\ope E_1 &= \sum_{a>0} \varphi_a (\Delta_{\R^n}-\Delta_{\mfd})(\lambda-\Delta_a)^{-1}\varphi_a,  \\
\ope E_2 &= -\sum_a [\Delta_{\mfd }, \varphi_a](\lambda-\Delta_{\R^n})^{-1}\varphi_a.
\end{align*}
We now proceed to estimate those errors.

Multiplication with $\varphi_a$ is a bounded operator on $h^{0,\alpha}$ with
\[
\|\varphi_a h_{jk}\|_{0,\alpha} \leq C\|h_{jk}\|_{0,\alpha},
\]
where the constant $C$ does not depend on our choice of $R\geq 1$ and $\epsilon$.  It then follows from \eqref{eq:nearly flat} that
\[
\|\varphi_a (\Delta_{\R^n}-\Delta_{\mfd })(\lambda-\Delta_{\mfd})^{-1}\varphi_a h_{jk}\|_{0,\alpha} \leq C\varepsilon \|h_{jk}\|_{0,\alpha}
\]
for all $h_{jk}\in h^{0,\alpha}(\sym\mfd)$.  Hence,
\[
\|\ope E_1 h_{jk}\|_{0,\alpha} = \Bigl\|\sum_{a>0} \varphi_a (\Delta_{\R^n}-\Delta_{\mfd })(\lambda-\Delta_{\R^n})^{-1}\varphi_a \cdot
h_{jk}\Bigr\|_{0,\alpha} \leq CN\varepsilon \|h_{jk}\|_{0,\alpha}.
\]
We can therefore make the first error as small as needed by reducing $\varepsilon$.

The second error operator $\ope E_2$ contains a commutator, which we now compute. For any $h_{jk}$, we have
\[ 
[\Delta_\mfd, \varphi_a] h_{jk} = [g^{pq}\nabla_p\nabla_q, \varphi_a] h_{jk}
= 2g^{pq} (\nabla_p\varphi_a)\nabla_qh_{jk} + g^{pq}(\nabla_p\nabla_q \varphi_a) h_{jk}.
\]
The bounds from Lemma~\ref{lem:cutoff grad bounds} on the derivatives of the cutoff functions $\varphi_a$ imply that
\[
\|[\Delta_\mfd, \varphi_a ]h_{jk} \|_{0,\alpha} \leq \frac{C}{R} \|h_{jk}\|_{1,\alpha}.
\]
Since $(\lambda-\Delta_{\R^n})^{-1}: h^{0,\alpha}\to h^{2,\alpha} \hookrightarrow h^{1,\alpha}$ is bounded, we can estimate the second error
operator $\ope E_2$ as follows:
\[
\forall h_{jk}\in h^{0,\alpha}(\sym\mfd): \qquad \|\ope E_2 h_{jk}\|_{0,\alpha} \leq \frac{CN}{R} \|h_{jk}\|_{0,\alpha}.
\]
The bounds for $\ope E_1$ and $\ope E_2$ together tell us that
\[
\|(\lambda-\Delta_\mfd)\ope S - \boldsymbol{1}\|_{h^{0,\alpha}\to h^{0,\alpha}} \leq CN \bigl(\varepsilon+R^{-1}\bigr).
\]
We can make this error as small as we like by choosing $\varepsilon$ small and $R$ large enough.  Therefore, $\ope S$ is a right inverse for
$\lambda-\Delta_\mfd$.  \medskip

To complete the proof, we must still show that $\ope S$ also is a left inverse, which we can do by verifying that $\lambda-\Delta_\mfd$ is
injective.  We can do this using the same arguments as in~\S\ref{sec:left inverse}.

\subsection{More on the essential spectrum} \label{MoreOnSpectrum}
We have shown that $\sigma(\Delta_\mfd) \subset (-\infty, 0]$, which
suffices for our purposes in this paper.  However, Assumption~\ref{Asymptotically flat} actually implies that $\sigma(\Delta_\mfd) = (-\infty, 0]$.
To prove this, let $\lambda=-k^2\in(-\infty, 0]$ be given, and choose a large number $R>0$ and a point $\mf p\in\mfd$ that is so far away
that $\mr{inj}(\mf p)>R$, while the metric in geodesic coordinates on $\ball_R(\mf p)$ is $\varepsilon$ close to the Euclidean metric.
Consider the function
\[
F(x) = \int_{\|\omega\|=k} e^{i\omega\cdot x} \,\mr d\mc S^{n-1}(\omega),
\]
in which $\mr d\mc S^{n-1}$ is the volume element on the $(n-1)$-sphere with radius $k$.  The function $F$ is smooth, and can be expressed
in terms of Bessel functions.  More importantly, it satisfies $\Delta_{\R^n}F + k^2 F = 0$.  One can verify that $F(x)\to0$ as
$|x|\to\infty$.

With this function in hand, we construct a near eigenfunction for $\Delta_\mfd$ with eigenvalue $-k^2$.  Again let $\eta:[0, \infty)\to\R$
be smooth with $\eta(t)=1$ for $t\leq \frac 12$, and $\eta(t)=0$ for $t\geq 1$. Consider the tensor field given in geodesic coordinates on
$\ball_R(\mf p)$ by
\[
h_{jk} = \eta\Bigl(\frac{\|x\|}{R}\Bigr)F(x)\,\delta_{jk}.
\]
Leaving some details to the reader, we claim that one can now verify that as $R\to\infty$ and $\varepsilon\to0$, one has
\[
\|h_{jk}\|_{2,\alpha} \geq C \quad\text{ and }\quad \bigl\|(\Delta_\mfd + k^2)h_{jk}\bigr\|_{0,\alpha}\to 0.
\]
This shows that $\Delta_\mfd + k^2: h^{2,\alpha}(\sym\mfd) \to h^{0,\alpha}(\sym\mfd)$ does not have a bounded inverse.

\subsection{A mollified distance function} \label{SmoothDistance}
Recall the partition of unity $\{\vp_a:a\in\mf A\}$ that we constructed in
\S~\ref{ResolventAtInfinity} and fix $\mf p_0\in K_{R,\ve}$.

For each $a\in\mf A$ with $a>0$, we define $r_a:B_{2R}(0)\subset\R^n \to [0,\infty)$ by
\[
r_a(\xi) = d(\exp_{\mf p_a}(\xi), \mf p_0).
\]
For $t>0$, we then let $r_a(t, \xi)$ be the solution of the initial boundary value problem
\[
\begin{aligned}
\frac{\partial r_a(t, \xi)}{\partial t} &= \Delta r_a(t, \xi) && \mbox{for }\xi\in B_{2R}(0)
&& \mbox{and }t\geq 0,  \\
r_a(0, \xi) &= r_a(\xi) && \mbox{for }\xi\in B_{2R}(0),  \\
r_a(t, \xi) &= r_a(\xi) && \mbox{for }\xi\in \partial B_{2R}(0) && \mbox{and }t\geq 0.
\end{aligned}
\]
The distance function $x\in\mfd\mapsto d(x, \mf p_0)$ is Lipschitz with constant $1$.  By~\eqref{Lip1} and~\eqref{Lip2}, the maps
$\exp_{\mf p_a}$ and $\exp_{\mf p_a}^{-1} $ are nearly isometries, hence Lipschitz, this implies that $r_a:B_{2R}\to[0,\infty)$ is Lipschitz
for some constant $C$.

We now claim that:
\begin{subequations}
\begin{align}
|r_a(1, x)-r_a(x)| &\leq C && x\in B_{2R}(0),  
    \label{eq ra close to initial ra}
\\
|\nabla r_a(1, x)| + |\nabla^2 r_a(1, x)| &\leq C && x\in B_{R}(0).
                           \label{eq ra always Lipschitz}
\end{align}  
\end{subequations}

The first inequality \eqref{eq ra close to initial ra} follows from the maximum principle.  Namely, for any given $x_0\in B_{2R(0)}$, we let
$r_\pm$ be the solution of the linear heat equation on $\R^n$ with initial data $r_\pm(0,x) = r_a(x_0)\pm C|x-x_0|$.  Self-similarity of the
heat equation implies that $r_{\pm}$ has the form
\[
r_\pm(t, x) = r_a(x_0) \pm C\sqrt{t}F\Bigl(\frac{x-x_0}{\sqrt{t}}\Bigr),
\]
for some function\footnote{Using the explicit solution of the heat equation, one finds an integral representation of $F$, namely
\[
F(\xi):=(4\pi)^{-n/2}\int e^{-\frac{1}{4}|\xi-z|^2}|z|\; dz.
\]
Derivation: the solution of $F_t=\Delta F$ with $F(0,\xi)= |\xi|$ is
\[
F(x,t) = \int e^{-\frac{|x-y|^2}{4t}} |y| \frac{dy}{(4\pi t)^{n/2}}
=
\int
e^{-\frac{1}{4}\left|\frac{x}{\sqrt{t}} - z\right|^2} |z|\;\frac{dz}{(4\pi t)^{n/2}} = \sqrt{t}\; F\Bigl(\frac{x}{\sqrt{t}}\Bigr),
\]
%%% \begin{align*}
%%% F(x,t) &= \int (4\pi t)^{-n/2} e^{-\frac{|x-y|^2}{4t}} |y| dy&& (y=x+\sqrt{t}\,z)\\
%%% &=\sqrt{t}\; (4\pi)^{-n/2} \int e^{-\frac{1}{4}\left|\frac{x}{\sqrt{t}} - z\right|^2} |z|\;dz\\, &= \sqrt{t}\;
%%%       F\Bigl(\frac{x}{\sqrt{t}}\Bigr),
%%%       \end{align*}
where we have substituted $y=x+\sqrt t\; z$ and defined $ F(\xi)=(4\pi)^{-n/2}\int e^{-\frac{1}{4}|\xi-z|^2}|z|\; dz $.  } $F$.
Since $r_a$ is Lipschitz with constant $C$, the maximum principle implies that
\[
r_-(t, x) \leq r_a(t, x) \leq r_+(t, x)
\]
for all $t\geq 0$, $x\in B_{2R}(0)$.  Setting $t=1$ and $x=x_0$, we get
\[
r_a(x_0) - CF(0) \leq r_a(1, x_0) \leq r_a(x_0)+ CF(0),
\]
which implies \eqref{eq ra close to initial ra}.

The second estimate \eqref{eq ra always Lipschitz} expresses interior derivative estimates for the linear heat equation.

We now define
\[
\rho(x) \stackrel{\rm def}= \sum_{a>0}\vp_a(x) r_a(1, \exp_{\mf p_a}^{-1}(x)).
\]
If $d(x, K_{R, \varepsilon})>2R$, then $\sum_{a>0}\vp_a(x) = 1$, and thus \eqref{eq ra close to initial ra} implies that
\[
|\rho(x) - d(x, \mf p_0)|\leq C.
\]
The partition functions $\vp_a$ all have uniformly bounded first and second derivatives, so the interior estimate \eqref{eq ra
  always Lipschitz} implies that $\nabla\rho$ and $\nabla^2\rho$ are uniformly bounded.

\subsection{Proof that $\sigma(\Delta_\ell) \subset \R$} 
In this section, we show that the little-H\"older and $L^2$ spectra of $\Delta_\ell$ coincide. In \S~\ref{SmoothDistance} above, we
have constructed a smooth function $\rho:\mfd\to[0,\infty)$ such that for all $x\in\mfd$,
\[
d(\mf p_0,x)-C \leq\rho(x)\leq d(\mf p_0,x)+C.
\]

Now we deal with $\Delta_\ell$. In \S~\ref{MoreOnSpectrum}, we have also shown that $\sigma_{\mr{ess}}(\Delta_\ell)=(-\infty,0]$. For
(small) $\theta\in\mb R$, consider the operator
\[
A(\theta) \stackrel{\rm def}= e^{\theta\rho}\,\Delta_\ell\,e^{-\theta\rho} = \Delta_\ell -2\theta\cv_{\cv\rho} +
\theta^2\big(\Delta_{\mfd}\rho+|\cv\rho|^2\big).
\]
Then $A(\theta):h^{2,\alpha}(\sym)\to h^{0,\alpha}(\sym)$ depends continuously on $\theta\in\mb R$ in the operator norm.

Suppose that $\lambda\in\mb C$ is an $h^{2,\alpha}(\sym)$ eigenvalue of $\Delta_\ell$ of multiplicity $m$, and choose $\delta>0$ small
enough so that $\sigma(\Delta_\ell)\cap B_\delta(\lambda)=\{\lambda\}$.  Denote the spectral projections of $\Delta_\ell$ and $A(\theta)$ by
\[
P_0 = \oint_{\partial B_\delta(\lambda)} \big(\mu-\Delta_\ell\big)^{-1}\,\frac{\mr d\mu}{2\pi i}\qquad\mbox{ and }\qquad P(\theta) =
\oint_{\partial B_\delta(\lambda)} \big(\mu-A(\theta)\big)^{-1}\,\frac{\mr d\mu}{2\pi i}\,,
\]
respectively. There exists $\theta^*$ small enough such that $P(\theta)$ depends continuously on $\theta\in(-\theta^*,\theta^*)$. Then,
because $P(\theta)$ has finite rank, that rank must be constant.

Let $\{u_1,\dots,u_m\}$ be a basis for $\mr{range}\big(P_0\big)$, \emph{i.e.,} the generalized eigenspace of $\Delta_\ell$ corresponding to
the eigenvalue $\lambda$. Then
\[
\{P(\theta)u_1,\dots,P(\theta)u_m\}
\]
is a basis for $\mr{range}\big(P(\theta)\big)$ for small $\theta$, and we have
\[
\sigma\big(A(\theta)\big)\cap B_\delta(\lambda) = \{\lambda_1,\dots,\lambda_k\}\quad\mbox{ for some }1\leq k \leq m.
\]

Now fix some $0<\theta_+<\theta^*$. We may assume that $\{u_1,\dots,u_m\}$ were chosen so that each $P(\theta_+)u_i$ is a generalized
eigenvector of $A(\theta_+)$, \emph{i.e.,} for each $i$ there exist $\lambda_{j_i}$ and $1\leq k_i \leq k$ such that
\[
\big(\lambda_{j_i}-A(\theta_+)\big)^{k_i} \big[P(\theta_+) u_i\big]=0.
\]
Recalling the definition of $A(\theta_+)$, this implies that
\[
e^{\theta_+\rho} \big(\lambda_{j_i}-\Delta_\ell\big)^{k_i}\big[e^{-\theta_+\rho} P(\theta_+) u_i\big]=0,
\]
which in turn implies that $e^{-\theta_+\rho} P(\theta_+) u_i$ are linearly independent generalized eigenfunctions of $\Delta_\ell$ with
eigenvalues $\lambda_{j_i}\in B_\delta(\lambda)$.  Since $\lambda$ is the only eigenvalue for $\Delta_\ell$ in $B_\delta(\lambda)$, it
follows that all $\lambda_{j_i}=\lambda$. Therefore, $\mr{range}\big(P_0\big)$ is spanned by the set of exponentially decaying
eigenfunctions
\[
\big\{e^{-\theta_+\rho}P(\theta_+)u_i:1\leq i \leq m\big\}.
\]

Because $|\Rm|\to0$ at spatial infinity, it follows from Bishop--Gromov volume comparison that $(\mfd,g)$ has subexponential volume
growth.  Therefore, every little-Hölder generalized eigenfunction of $\Delta_\ell$ is an $L^2(\sym)$ eigenfunction.  The converse holds
by elliptic regularity. And since $\Delta_\ell$ is self-adjoint in $L^2(\sym)$, there are only real eigenvalues.

\medskip This completes the proof of Main Theorem~\ref{MT1}.

\section{The spectrum of $\Delta_\ell$ on doubly warped products} \label{Lichnerowicz}

In this section, we prove:

\begin{prop} \label{UnstableSpectrum}
If $4\leq n_1+n_2\leq8$, then on any of Böhm's Ricci-flat manifolds $\big(\mfd^{1+n_1+n_2},g_{\mr B}\big)$, constructed as a doubly
warped product metric on $\mb R_+\times\mc S^{n_1}\times\mc S^{n_2}$, the Lichnerowicz Laplacian has infinitely many positive
eigenvalues.
\end{prop}

We then show in \S~\ref{Simple Evals} that these eigenvalues are simple.

\subsection{Doubly warped product geometries}
To begin, we consider any manifold $\mfd = \mb R^{n_1+1} \times \mc S^{n_2}$.  The group
$\mf G=\mathrm{O}(n_1+1)\times\mathrm{O}(n_2+1)$ acts on $\mfd $, and Ricci flow is equivariant with respect to the group
action. So, we may specialize to Ricci flow of $\mf G$-invariant metrics.

Such metrics have the form reviewed in Appendix~\ref{Geometric data}: doubly warped products over the base $[0,\infty)$ with fibers
$\mfd _1 = \mc S^{n_1},\;\mfd _2 = \mc S^{n_2}$, standard round fiber metrics $\hat g_1=g_{\mc S^{n_1}},\;\hat g_2=g_{\mc S^{n_2}}$, and
warping functions $v_0(x),v_1(x),v_2(x)$.

Specifically, we may without loss of generality take $v_0\equiv1$ to consider a fixed background metric
\begin{equation} \label{Background} g_{\mr B} = \mr dx^2 + v_1^2(x)\,\hat g_1+v_2^2(x)\,\hat g_2
\end{equation}
and linearize along a $\mf G$-invariant tensor field of the form
\begin{equation} \label{DWform} h=\eta_0(x)\,\dx^2+\eta_1(x)\,\hat g_1+\eta_2(x)\,\hat g_2
\end{equation}
on $(\mfd , g_{\mr B})$.  Although the results in this section generalize readily to any Ricci-flat metric, we are most interested in the
case that $g_{\mr B}$ is an Einstein metric constructed by Böhm~\cite{B99} and reviewed in~\cite{GAFA}.  \medskip

To begin, we employ Roman indices on $\mc S^n_1$, Greek indices on $\mc S^{n_2}$, and proceed to compute
$\Delta_\ell h = \Delta h +2\Rm*h -2\Rc*h$.  \medskip

The curvature formulas derived in Appendix~\ref{Geometric data} imply easily that
\begin{align*}
(\Rm*h)_{00}	&=\Big\{-n_1\frac{v_{1,xx}}{v_1}\eta_1-n_2\frac{v_{2,xx}}{v_2}\eta_2\Big\},\\
(\Rm*h)_{ij}		&=\Big\{-\frac{v_{1,xx}}{v_1}\eta_0+(n_1-1)\frac{1-v_{1,x}^2}{v_1^2}\eta_1-n_2\frac{v_{1,x}v_{2,x}}{v_1v_2}\eta_2\Big\}g_{ij},\\
(\Rm*h)_{\alpha\beta}	&=\Big\{-\frac{v_{2,xx}}{v_2}\eta_0-n_1\frac{v_{1,x}v_{2,x}}{v_1v_2}\eta_1+(n_2-1)\frac{1-v_{2,x}^2}{v_2^2}\eta_2\Big\}g_{\alpha\beta}.
\end{align*}
Because $g_{\mr B}$ is Ricci-flat, the terms $\Rc*h$ do not appear.

With some additional work, the Hessian derivation in Appendix~\ref{Tensor Hessian} implies that
\begin{align*}
(\Delta h)_{00}	&= \Big\{\eta_{0,xx}+n_1\frac{v_{1,x}}{v_1}\big[\eta_{0,x}+2\frac{v_{1,x}}{v_1}(\eta_1-\eta_0)\big]
+n_2\frac{v_{2,x}}{v_2}\big[\eta_{0,x}+2\frac{v_{2,x}}{v_2}(\eta_2-\eta_0)\big]\Big\},\\
(\Delta h)_{ij}	&=\Big\{\eta_{1,xx}+\big(n_1\frac{v_{1,x}}{v_1}+n_2\frac{v_{2,x}}{v_2}\big)\eta_{1,x}+2\frac{v_{1,x}^2}{v_1^2}(\eta_0-\eta_1)\Big\}g_{ij},\\
(\Delta h)_{\alpha\beta}	&=\Big\{\eta_{2,xx}+\big(n_1\frac{v_{1,x}}{v_1}+n_2\frac{v_{2,x}}{v_2}\big)\eta_{2,x}+2\frac{v_{2,x}^2}{v_2^2}(\eta_0-\eta_2)\Big\}g_{\alpha\beta},\\
\end{align*}

\subsection{The linearized system}  \label{sec: the linearized system}
The equations above are simplified somewhat by the change of variable
\begin{equation} \label{eq:define w} w_a \stackrel{\rm def} = \log \Big(\frac{v_a}{\sqrt{n_a-1}}\Big)\quad\mbox{for}\quad a=1,2.
\end{equation}
Then by representing $h$ as a vector $\eta=\mat \eta_0\\\eta_1\\\eta_2\rix$, the Lichnerowicz Laplacian becomes
\begin{equation}	\label{Vector Lichnerowicz}
\Delta_\ell \eta = \eta_{xx} +\mc A\eta_x + \mc B\eta,
\end{equation}
where
\begin{equation}	\label{A in LichLap defined}
\mc A(x)=\sum_a n_a w_{a,x}
\end{equation}
and
\begin{equation}
\label{B in LichLap defined}
\mc B(x) = \mat
-2n_1w_{1,x}^2-2n_2w_{2,x}^2		& -2n_1w_{1,xx}				& -2n_2w_{2,xx}				\\
-2w_{1,xx}						& 2e^{-2w_1}-2n_1w_{1,x}^2		& -2n_2w_{1,x}w_{2,x}			\\
-2w_{2,xx} & -2n_1w_{1,x}w_{2,x} & 2e^{-2w_2}-2n_2w_{2,x}^2 \rix.
\end{equation}

While it is a standard fact that $\Delta_\ell$ is self-adjoint in $L^2$, it is useful to verify this directly. Indeed, if $h$ and $\tilde h$ both
have the form~\eqref{DWform}, then their pointwise inner product is
\[
\lp h,\tilde h\rp_{g_{\mr B}} = \eta_0\tilde\eta_0 + n_1\eta_1\tilde\eta_1 + n_2\eta_2\tilde\eta_2,
\]
which one can write as
\[
\lp h,\tilde h\rp_{g_{\mr B}} = \eta^T \mc W\, \tilde\eta,\qquad\mbox{with the inner product}\qquad
\mc W =\left(
\begin{smallmatrix}
1 & & \\ & n_1 & \\ & & n_2
\end{smallmatrix}
\right).
\]
It follows that there exists a positive constant $\omega_{n_1,n_2}$ such that
\begin{equation}				 \label{L2 inner product}
\omega_{n_1,n_2}\big(h,\tilde h\big)_{L^2} = \int_0^\infty \eta^T \mc W\,
\tilde\eta\;v_1^{n_1}(x)\,v_2^{n_2}(x)\,\mr dx.
\end{equation}
It is then easy to see that $\big(h,\Delta_\ell h\big)_{L^2} = \big(\Delta_\ell h, h\big)_{L^2}$, once one observes that
\[
\mc W\mc B = \mat
-2n_1w_{1,x}^2-2n_2w_{2,x}^2		& -2n_1w_{1,xx}				& -2n_2w_{2,xx}				\\
-2n_1w_{1,xx}					& 2n_1(e^{-2w_1}-n_1w_{1,x}^2)	& -2n_1n_2w_{1,x}w_{2,x}		\\
-2n_2w_{2,xx} & -2n_1n_2w_{1,x}w_{2,x} & 2n_2(e^{-2w_2}-n_2w_{2,x}^2) \rix
\]
is symmetric.

\subsection{The Lichnerowicz Laplacian on a Ricci-flat cone}
To guide what follows, we consider the (singular) Ricci-flat cone metric to
which a given Böhm solution is asymptotic. We show in~\cite{GAFA} that $g_{\mr B}(x)\sim g_{\rfc}$ as $x\to\infty$, where
\[
g_{\rfc}=\mr dx^2+c_1^2x^2\,\hat g_1+c_2^2x^2\,\hat g_2,\qquad\mbox{with}\qquad c_a^2=\frac{n_a-1}{n-1},\quad (a=1,2).
\]
Then $v_a=\sqrt{c_a}\,x$, and, writing $n=n_1+n_2$, one finds on the cone that
\begin{equation} \label{ConicalLinearization} \Delta_\ell\eta = \eta_{xx}+\frac{n}{x}\eta_x+\frac{1}{x^2} \mc C\eta,
\end{equation}
where
\[
\mc C=\lim_{x\to\infty} \mc B(x) =2\mat
-n & n_1 & n_2 \\
1 & n_2-1 & -n_2\\
1 & -n_1 & n_1-1 \rix.
\]

The eigenvalues of $\mc C$ are $0,-2(n+1),2(n-1)$, and its  eigenvectors are
\[
%left eigen vectors from before:
%&\ell_0 = \mat 1 & n_1 & n_2 \rix, \qquad \ell_{-2(n+1)} = \mat -n & n_1 & n_2 \rix,\qquad \ell_{2(n-1)}=\mat 0 & 1 & -1\rix,\\
r_0 = \mat 1 &1&1\rix^\top,\qquad
r_{-2(n+1)}=\mat -n & 1 & 1\rix^\top,\qquad
r_{2(n-1)}=\mat 0 & n_2 & -n_1\rix^\top.
\]

Thus, if $\eta = \psi_0 r_0 + \psi_1 r_{-2(n+1)} + \psi_2 r_{2(n-1)}$ and \(\psi(x) = \smat \psi_0(x) \\ \psi_1(x) \\ \psi_2(x)\srix\),
then the action of \(\Delta_\ell\), given by~\eqref{ConicalLinearization}, becomes
\[
\Delta_\ell\psi= \psi_{xx}+\frac{n}{x}\psi_x + \frac{1}{x^2} \smat 0 & & \\ & -2(n+1) & \\ & & 2(n-1)\srix\psi.
\]

\subsection{The Lichnerowicz Laplacian on a Böhm Ricci-flat metric}		\label{x=0 asymptotics}

Now we return to a given Böhm Ricci-flat metric $g_{\mr B}$. We see from \S~7 of~\cite{GAFA} that
\begin{equation} \label{Convergence to rfc} v_a(x) =\sqrt{\frac{n_a-1}{n-1}}\,x+o(x) \quad\Rightarrow\quad
w_a(x)=\log\left(\frac{x+o(x)}{\sqrt{n-1}}\right)\qquad (x\to\infty).
\end{equation}
In the variables used here, we rigorously obtain convergence of the derivatives $w_{a,x}$ and $w_{a,xx}$ matching what one would obtain
by formally differentiating~\eqref{Convergence to rfc}, namely
\[
w_{a,x}=\frac{1}{x}+o\big(\frac{1}{x}\big),\qquad w_{a,xx} = -\frac{1}{x^2}+o\big(\frac{1}{x^2}\big),\qquad (x\to\infty).
\]
\medskip

The action of $\Delta_\ell$ on tensors of the form $h= \eta_0\mr dx^2 + \eta_1\hat g_1 + \eta_2\hat g_2$ is given
by~\eqref{Vector Lichnerowicz}, \textit{i.e.,} by the ordinary differential operator
\[
\ope D \stackrel{\rm def}= \partial_x^2 + \mc A(x)\partial_x + \mc B(x).
\]
The asymptotic behavior of the metric as $x\to\infty$ implies that 
\[
\ope D = \partial_x^2 + \Big(\frac{n}{x}+o(x^{-1})\Big)\partial_x
+\frac{1}{x^2}\mc C +o(x^{-2}).
\]

For functions $\vp(x)$ to be determined, we consider vectors of the form
\[
\eta(x) = \vp(x)\,r_{2(n-1)}= \vp(x)\mat 0\\n_2\\-n_1\rix.
\]
Neglecting the error terms in $\ope D$, the exact solutions of
\[
\vp''(x) + \frac{n}{x} \vp'(x) + \frac{2(n-1)}{x^2}\vp(x)=\frac{\mu}{x^2}\vp(x)
\]
are
\[
\vp_\pm(x) = x^{-\frac{n-1}{2}\pm\frac12\sqrt{n^2-10n+9+4\mu}}.
\]
For all $4\leq n \leq 8$ and all $\mu\in(0,7/4)$, one has $n^2-10n+9+4\mu<0$. Accordingly, we set
\[
\alpha \stackrel{\rm def} = \frac{1-n}{2}\qquad\mbox{and}\qquad\beta \stackrel{\rm def} =\sqrt{|n^2-10n+9+4\mu|}
\]
in order to write the general solution $\vp(x) = c_1\vp_1(x)+c_2\vp_2(x)$ in terms of
\[
\vp_1(x) \stackrel{\rm def} = x^\alpha\, \sin(\beta\log x)\qquad\mbox{and}\qquad \vp_2(x) \stackrel{\rm def} = x^\alpha\, \cos(\beta\log x).
\]
This is the dimension restriction explanation that we promised in the Introduction.

\subsection{Approximate eigenvectors}
We are now ready to accomplish the first goal of this section.

\begin{proof}[Proof of Proposition~\ref{UnstableSpectrum}]
For any large constant $R>0$, we consider an approximate eigenvector $\eta(x)=\vp(x)\,r_{2(n-1)}$ defined with respect to the compactly
supported test function
\[
\vp(x) \stackrel{\rm def} =
\left\{\begin{matrix} \vp_1\big(\frac{x}{R}\big) & R\leq x \leq e^{\pi/\beta}R, \\ \\
0 & \mbox{elsewhere.}  \end{matrix}\right.
\]
Then by~\eqref{L2 inner product}, representing $h$ by $\eta$, one has
\begin{align*}
\omega\big(\eta,\ope D[\eta]\big)_{L^2} &= \int_0^\infty \eta^T \mc W\, \ope D[\eta]\;v_1^{n_1}(x)\,v_2^{n_2}(x)\,\mr dx\\
                  &= \int_R^{e^{\pi/\beta}R} \eta^T \mc W\Big\{\frac{\mu}{x^2}\eta+o\big(x^{-1}\eta_x+x^{-2}\eta\big)\Big\}\;v_1^{n_1}(x)\,v_2^{n_2}(x)\,\mr dx.
\end{align*}
On the interval $x\in[R,e^{\pi/\beta}R]$, one has estimates $|\vp(x)|\leq C$ and $|\vp'(x)|\leq C$, and hence has
\[
\Big|\frac{\eta(x)}{x^2}\Big|\leq \frac{\big|\vp\big(\frac{x}{R}\big)\big|}{R^2} \leq\frac{C}{R^2}\qquad\mbox{and}\qquad
\Big|\frac{\eta_x(x)}{x^2}\Big|\leq\frac{1}{x}\frac{1}{R}\Big|\vp'\big(\frac{x}{R}\big)\Big|\leq\frac{C}{R^2},
\]
where $C$ depends only on $n_1,n_2$. Therefore, as $R\to\infty$, we have
\[
\omega\big(\eta,\ope D[\eta]\big)_{L^2} =\mu \int_R^{e^{\pi/\beta}R} \frac{\big(1+o(1)\big)}{x^2}\big(\eta^T \mc
W\eta\big)\;v_1^{n_1}(x)\,v_2^{n_2}(x)\,\mr dx>0.
\]

Taking a sequence $R_k\to\infty$ with $R_{k+1}\geq e^{\pi/\beta}R_k+1$, one obtains in this way a family $h_k$ such that
$\big(h_k,\Delta_\ell h_k\big)_{L^2}>0$.  Because the $h_k$ have mutually disjoint supports, they form an orthogonal sequence. This implies
that $\Delta_\ell$ has infinitely many positive eigenvalues.
\end{proof}

\subsection{Simple eigenvalues}		\label{Simple Evals}

Our second goal in this section is to prove that the eigenvalues of $\Delta_\ell$ with $\mf G$-invariant eigentensors are simple. The proof is in three steps.
We begin with a smoothness criterion.

\begin{lemma}  		\label{lemma:boundary condition at zero}
 If the $(2,0)$ tensor field on $\R^{n_1+1}\times\mc S^{n_2}$ given by
 \[
 h = \eta_0(x)(\mr dx)^2 + \eta_1(x)\hat g_1 + \eta_2(x)\hat g_2
 \]
 is smooth and  $\mf G = \mr{O}(n_1+1)\times\mr{O}(n_2+1)$-invariant, then
 \begin{equation}
   \label{eq:boundary condition at zero}
      \lim_{x\searrow 0} \frac{\eta_1(x)}{x^2} = \lim_{x\searrow 0} \eta_0(x).
 \end{equation}
\end{lemma}

\begin{proof}
  Let $X=(x_0, \dots, x_{n_1})$ be coordinates on $\R^{n_1+1}$, and set
  \[
  x=|X|=\sqrt{x_0^2+\cdots+x_{n_1}^2}.
  \]
  The only $\mr{SO}(n_1+1)$-invariant $(2,0)$ tensor fields on $\R^{n_1+1}$ are linear combinations of $|\mr dX|^2$ and $\langle X, \mr dX\rangle^2$, \emph{i.e.,}
  \[
    h =  \phi(x)|\mr dX|^2 + \psi(x) \langle X, \mr dX\rangle^2
  \]
  for smooth even functions $\phi,\psi:\R\to\R$.  The Euclidean metric on $\R^{n_1+1}$ is
  \[
    |\mr dX|^2 = (\mr dx)^2 + x^2 \hat g_1.
  \]
  Furthermore, $x^2 = |X|^2$ implies that $x\,\mr dx = \langle X, \mr dX\rangle$.  Therefore, the tensor $h$ may be written as
  \[
    h = \bigl(\phi(x)+x^2\psi(x)\bigr)(\mr dx)^2 +  x^2\phi(x) \hat g_1.
  \]
  This implies \eqref{eq:boundary condition at zero}
\end{proof}

Now let $\lambda_0\geq \lambda_1\geq \lambda_2\geq \cdots>0$ be the sequence of positive eigenvalues for $\Delta_\ell$ with eigentensors
$h_j$ whose existence we have established above. Our second step is as follows.

\begin{lemma}			\label{lemma:eta2 not zero}
For all $j\geq0$, the $\eta_2(0)$ component of $h_j$ is never zero.
\end{lemma}

\begin{proof}
In \S~\ref{sec: the linearized system} above, we represented the Lichnerowicz Laplacian as an ordinary differential operator
\[
\Delta_\ell \eta = \eta_{xx} +\mc A(x)\eta_x + \mc B(x)\eta,
\]
where $\mc A(x)$ and $\mc B(x)$ are defined in terms of the warping functions $v_1$, $v_2$ as in \eqref{A in LichLap defined} and \eqref{B in LichLap defined}.  
Both $\mc A$ and $\mc B$ are singular at the origin, but $x\mc A(x)$ and $x^2\mc B(x)$ are real analytic functions of $x^2$.  Indeed,
the warping functions themselves are real analytic functions, where $v_1$ is odd with $v_1'(0)=1$, while $v_2$ is even.
Thus,
\[
v_1(x) = x + \sum_{k=1}^\infty \frac{v_1^{(2k+1)}(0)}{(2k+1)!}x^{2k+1}\qquad\mbox{and}\qquad
v_2(x) = v_2(0) + \sum_{k=1}^\infty \frac{v_2^{(2k)}(0)}{(2k)!} x^{2k}.
\]
Hence we have
\[
  w_{1, x} = \frac{1}{x}+\mc O(x), \quad w_{1,xx} = -\frac{1}{x^2} + \mc O(1),\quad
  w_{2,x} = \mc O(x), \quad w_{2,xx} = \mc O(1),
\]
and 
\[
  e^{-2w_1} = \frac{n_1-1}{v_1^2} = \frac{n_1-1}{x^2}+\mc O(1),\qquad
  e^{-2w_2} = \frac{n_2-1}{v_2^2} = \mc O(1),
\]
where $\mc O(1)$ and $\mc O(x)$ represent even and odd power series in $x$, respectively.
We also note that $w_{1,x}w_{2,x} = \mc O(1)$.

In this way, we get the following expansions for $\mc A(x)$ and $\mc B(x)$:
\[
 x \mc A(x) =\sum_a n_a \; x w_{a,x} =  n_1 + \mc O(x^2)
\]
and
\begin{align*}
x^2\mc B(x)
  &= x^2\mat
		-2n_1w_{1,x}^2-2n_2w_{2,x}^2	& -2n_1w_{1,xx}				& -2n_2w_{2,xx}				\\
		-2w_{1,xx}						& 2e^{-2w_1}-2n_1w_{1,x}^2	& -2n_2w_{1,x}w_{2,x}		\\
		-2w_{2,xx} & -2n_1w_{1,x}w_{2,x}	& 2e^{-2w_2}-2n_2w_{2,x}^2
		\rix   \\
&=  \mc B_0+ \mc O(x^2),
\end{align*}
with
\[
\mc B_0 =  \mat
-2n_1		& 2n_1				& 0			\\
2						& 2(n_1-1)-2n_1		& 0		\\
0& 0 & 0 \rix 
=  \mat
-2n_1	& 2n_1	& 0		\\
2		& -2		& 0		\\
0		& 0 		& 0 		\rix .
\]
It follows that $x=0$ is a regular singular point for the eigentensor equation
\begin{equation}
\label{eigen tensor}
x^2\eta''(x)+x\mc A(x) x\eta'(x) + x^2\mc B(x)\eta(x) = \lambda x^2\eta(x).
\end{equation}
Because the coefficients $x\mc A(x)$ and $x^2\mc B(x)$ are real analytic functions of $x^2$ in a neighborhood of the origin of the complex
plane, $x=0$ is a regular singular point for the eigentensor equation \eqref{eigen tensor}.  One can therefore use the Fröbenius method to find a fundamental set of
solutions.  These solutions are of the form
\[
\eta(x) = x^\alpha\sum_{k=0}^\infty a_kx^{2k}\quad
\text{ or }\quad
\eta(x) = x^\alpha\sum_{k=0}^\infty \bigl\{a_kx^{2k} + b_kx^{2k}\log x\bigr\},
\]
for suitable characteristic exponents $\alpha\in\C$, and coefficients $a_k, b_k\in\C^3$.  The second form appears in case two
characteristic exponents differ by an integer.

To find the characteristic exponents, we consider the principal terms of the eigenvalue equation, which are
\[
 x^2 \zeta''(x) +  n_1 \zeta'(x) +    \mc B_0\zeta(x)= 0.
\]
The eigenvalues of $\mc B_0$ are
\begin{align*}
\beta_0 	&= -(2n_1+1),	&
\beta_1	&= 0,			&
\beta_2 	&= 0,
\end{align*}
with respective eigenvectors
\begin{align*}
\zeta_0 	&= \mat 2n_1\\ -1\\0 \rix,	&
\zeta_1 	&=\mat 1\\1\\0 \rix, 		&
\zeta_2 	&= \mat 0\\ 0\\1 \rix.
\end{align*}

To obtain the characteristic exponents, we look for solutions having the form $\zeta(x) = x^{\alpha_j}\zeta_j$ for $\zeta_j\in\R^3$.
We find for all $j$ that
\[
\alpha_j(\alpha_j-1)+n_1\alpha_j + \beta_j = 0,
\]
\textit{i.e.,}
\begin{gather*}
\alpha_0(\alpha_0-1)+n_1\alpha_0-2(n_1+1)=0, \\
\alpha_1(\alpha_1-1)+n_1\alpha_1 = \alpha_2(\alpha_2-1)+n_1\alpha_2=0.
\end{gather*}
Solving the quadratic equations leads to the following six characteristic exponents:
\[
  \alpha_0   =+2, \quad
  \hat\alpha_0 = -(n_1+1), \quad
  \alpha_1=\alpha_2 = 0, \quad
  \hat\alpha_1=\hat\alpha_2=-(n_1-1).
\]
The leading terms of the six corresponding fundamental solutions $\zeta_j(x), \hat\zeta_j(x)$ are
\begin{alignat*}{6}
\zeta_0(x) &=   x^2\zeta_0 + o(x^2) &&= \mat 2n_1x^2+o(x^2) \\ -x^2+o(x^2)\\ o(x^2)\rix; \qquad&
&\hat\zeta_0(x) = x^{-n_1-1}\zeta_0+ o\bigl(x^{-n_1-1}\bigr); \\
\zeta_1(x)&=\zeta_1+ o(1) &&= \mat 1+o(1)\\1+o(1)\\o(1)\rix; &
&\hat\zeta_1(x)=x^{-n_1+1}\zeta_1 + o\bigl(x^{-n_1+1}\bigr); \\
  \zeta_2(x)&=\zeta_2+o(1) &&=\mat o(1)\\o(1)\\1+o(1)\rix;  &
&\hat\zeta_2(x)=x^{-n_1+1}\zeta_2 +o\bigl(x^{-n_1+1}\bigr).
\end{alignat*}
Since any eigentensor $h=\eta_0(x)(\mr dx)^2+\eta_1(x)\hat g_1+\eta_2(x)\hat g_2$ is a linear combination of these fundamental
solutions, its behavior as $x\searrow0$ is given by one of the expressions above.  

The fundamental solutions $\hat\zeta_j(x)$ are excluded because they possess negative characteristic exponents (remember that $n_1\geq 2$) and
therefore are unbounded at $x=0$.  

The fundamental solutions $\zeta_0(x)$ and $\zeta_1(x)$ are also excluded, because they do not satisfy the smoothness criterion of
Lemma~\ref{lemma:boundary condition at zero}.  

It follows that the only possible behavior of an eigentensor at $x=0$ is given by~$\zeta_2(x)$.  This implies that $\eta_2(0)\neq 0$.
\end{proof}

Now our second main result of this section follows as an easy corollary:

\begin{prop}			\label{prop:eigenvalues simple}
The eigenvalues $\lambda_j>0$ of $\Delta_\ell$ are simple.
\end{prop}

\begin{proof}
Let
\[
 h_j = \eta_0(x)(\mr dx)^2 + \eta_1(x)\hat g_1 + \eta_2(x)\hat g_2,
\qquad
 \tilde h_j = \tilde\eta_0(x)(\mr dx)^2 + \tilde\eta_1(x)\hat g_1 + \tilde\eta_2(x)\hat g_2
\]
be nonzero eigentensors of $\Delta_\ell$ for the same eigenvalue $\lambda_j$.  Then we can see from
Lemma~\ref{lemma:eta2 not zero} that $\eta_2(0)\neq 0$ and $\tilde\eta_2(0)\neq 0$.  Consider the eigentensor
\[
\hat h_j =\eta_2(0)\tilde h_j - \tilde\eta_2(0) h_j,
\]
which we can write as
\[
 \hat h = \hat\eta_0(x)(\mr dx)^2 + \hat\eta_1(x)\hat g_1 + \hat\eta_2(x)\hat g_2.
\]
Since $\hat \eta_2(0)=\eta_2(0)\tilde \eta_2(0) - \tilde\eta_2(0) \eta_2(0)=0$, it follows from Lemma~\ref{lemma:eta2 not zero}
that $\hat h_j=0$, and hence that $h_j$ and $\hat h_j$ are linearly dependent.
\end{proof}

\section{Distinct instabilities of the nonlinear system}		\label{sec:Distinct}
For every solution $g$ in Main Theorem~\ref{MT1}, there exists a DeTurck vector field $X_g$.  The metric $g(t)$ is then a solution of
Ricci--DeTurck flow,
\[
\partial_t g =-2\Rc[g] + \mc L_{X_g}g.
\]
A family of diffeomorphisms $\phi_t:\mfd\to\mfd$ is a DeTurck family for the solution $g(t)$ if
\[
\forall p\in \mfd,\; t\leq 0:\qquad \frac{\partial\phi_t(p)}{\partial t} = X_g(\phi_t(p)) .
\]
Such a family exists with the property that $\phi_0(p)=p$ for all $p\in\mfd$.  The following paragraph shows that we can choose $\phi$ so
that $\phi_t(p)\rightarrow p$ as $t\to-\infty$.

The fact that the metrics $g(t)$ converge exponentially to $g$ in backward time implies that the DeTurck vector field $X_g$ decays exponentially
as $t\to-\infty$.
This
implies that the DeTurck diffeomorphisms $\phi_t:\mfd\to\mfd$ converge at an exponential rate to a fixed map that must itself be a diffeomorphism.
Label it $\phi_{-\infty}:\mfd\to\mfd$, and let $\check\phi_t = \bigl(\phi_{-\infty}\bigr)^{-1}\circ \phi_t$.  Then the diffeomorphisms $\check\phi_t$ also
evolve by the DeTurck vector field, and for $t\to-\infty$ they converge to the identity map.  It follows that
\[
\check g(t) = \check\phi_t^* g(t)
\]
is a solution of Hamilton's Ricci flow with $g(t)\to g$ as $t\to-\infty$.

While our construction guarantees that the solutions of Ricci--DeTurck flow in the $N_\delta$-dimensional family are distinct, it does
not immediately follow that the corresponding solutions $\check g(t)$ of Ricci flow are geometrically distinct: imagine that
$g_1(t), g_2(t)$ are two solutions of the Ricci--DeTurck flow, and suppose that their corresponding Ricci flows
\(\check g_1(t), \check g_2(t) \) are equivalent, \emph{i.e.,} suppose there exists a diffeomorphism $\psi:\mfd\to\mfd$ such that
$\check g_2(t) =\psi^*[\check g_1(t)]$ for all \(t\).

We show below that this cannot happen under the assumption that \emph{the Ricci-DeTurck solutions \(g_1(t), g_2(t)\) both are invariant
  under the isometry group \(\mf G\) of \((\mfd, g)\).}  This assumption holds in the case of doubly warped solutions emanating from
the Böhm soliton that appear in Main Theorem~\ref{MT2}, for in this case the full isometry group of the Ricci-flat soliton is
\(\mf G =\mr{O}(n_1+1)\times\mr{O}(n_2+1)\), and the solutions we construct all have the same symmetry.

By construction, we have $\check g_{1,2}(t) \to g$ as $t\to-\infty$.  This implies that $g=\psi^*g$, \emph{i.e.,} that $\psi$ belongs
to the isometry group \(\mf G\) of $g$.

Let $\check\phi_{t,i}:\mfd\to\mfd$ be the DeTurck diffeomorphisms with $\check \phi_{t, i}\to\mr{id}_\mfd$ as $t\to-\infty$.  Because
the Ricci-flat metric \(g\) and the DeTurck metrics \(g_i(t)\) are \(\mf G\)-invariant, their DeTurck vector fields and diffeomorphisms
are also \(\mf G\)-invariant.  Therefore they commute with the isometry $\psi$, which implies
\[
\check g_2(t) = \psi^*[\check g_1(t)] 
=\psi^*[\check\phi_{t,1}^*g_1(t)]
=\check\phi_{t, 1}^*\bigl[\psi^* g_1(t)\bigr]
=\check\phi_{t, 1}^*\bigl[ g_1(t)\bigr]
=\check g_1(t).
\]
Therefore, the two solutions of Ricci--DeTurck flow are the same. It follows that if solutions of Ricci--DeTurck flow are pairwise distinct,
then the same is true for the corresponding solutions of Ricci flow.
\medskip

This argument completes the proof of Main Theorem~\ref{MT2}.

\appendix

\section{Review of doubly warped product geometries}	\label{Geometric data}

Both the background metrics $g_{\mr B}$ and the perturbed metrics $\tilde g$ we consider here are examples of doubly warped product metrics.
In this Appendix, we recall basic facts about the geometry of such metrics, suppressing decorations throughout.
\medskip

Given Riemannian manifolds $(\mfd _1^{n_1},\hat g_1)$ and $(\mfd _2^{n_2},\hat g_2)$ and an interval $\mc I\subseteq\mb R$,
one may form $(\mfd ^{n+1},g)$, where $\mfd =\mc I\times\mfd _1\times\mfd _2$, $n=n_1+n_2$, and
\begin{subequations}\label{eq:DWP Ansatz}
        \begin{align}
                g	&= \ds^2 + g_1 + g_2\\
                        &=v_0^2\, \dx^2 + v_1^2\,\hat g_1+v_2^2\,\hat g_2.
        \end{align}
\end{subequations}
Specifically, for increased generality, we take $x\in\mb R_+$ as a fixed coordinate but do not assume here that $v_0(x)$ is constant.

\subsection{The Levi--Civita connection}		\label{DWP connection}

We compute Christoffel symbols  inductively. In the first step, temporarily set $\bar g=\ds^2+g_1$.
We compute with respect to the vector field
\[
e_0=\frac{\partial}{\partial s} = v_0^{-1}\frac{\partial}{\partial x},
\]
and use indices $1\leq i,j,k\leq n_1$. We find that the connection of $\bar g$ is determined by
\begin{align*}
\Gamma_{00}^0	&=	\frac{v_{0,x}}{v_0}=v_{0,s},\\
\Gamma_{00}^k 	&= \Gamma_{i0}^0	=0,\\
\Gamma_{ij}^0	&=-\frac12\big(\Ds (v_1^2)\big)(\hat g_1)_{ij}, \qquad\qquad \Gamma_{0j}^k	=\frac12 u_1^{-1}\big(\Ds(v_1^2)\big)\delta_j^k,\\
\Gamma_{ij}^k	&= (\hat\Gamma_1)_{ij}^k.
\end{align*}

In the second step, we note that $g=\bar g+g_2$. We temporarily let $0\leq i,j,k\leq n_1$ and $n_1+1\leq\alpha,\beta,\gamma\leq n_1+n_2=n$.
We find that the connection of $g$ is given by
\begin{align*}
\Gamma_{00}^0&\mbox{ and }\Gamma_{ij}^k	\qquad\mbox{ as above},\\
\Gamma_{ij}^\gamma		&=\Gamma_{\alpha j}^k=0,\\
\Gamma_{\alpha\beta}^k	&=-\frac12 g_*^{k\ell} \big(\partial_\ell(v_2^2)\big))(\hat g_2)_{\alpha\beta},
\qquad\qquad \Gamma_{i\beta}^\gamma = \frac12 u_2^{-1}\big(\partial_i (v_2^2)\big)\delta_\beta^\gamma,\\
\Gamma_{\alpha\beta}^\gamma	&= (\hat\Gamma_2)_{\alpha\beta}^\gamma.
\end{align*}

In the final step, using $e_0=\ps$, $1\leq i,j,k\leq n_1$, and $n_1+1\leq\alpha,\beta,\gamma\leq n_1+n_2$, we combine the results above
to see that the only potentially nonzero symbols are
\begin{subequations}
\begin{align}
\Gamma_{00}^0 	&= v_{0,s},\\
\Gamma_{ij}^0	&= -\frac{v_{1,s}}{v_1}\,(g_1)_{ij},\qquad\qquad
\Gamma_{0j}^k	=\frac{v_{1,s}}{v_1}\,\delta_j^k,\\
\Gamma_{ij}^k	&= (\hat\Gamma_1)_{ij}^k,\\
\Gamma_{\alpha\beta}^0	&=-\frac{v_{2,s}}{v_2}\,(g_2)_{\alpha\beta},\qquad\qquad
\Gamma_{0\beta}^\gamma=\frac{v_{2,s}}{v_2}\,\delta_\beta^\gamma,\\
\Gamma_{\alpha\beta}^\gamma	&=(\hat\Gamma_2)_{\alpha\beta}^\gamma.
\end{align}
\end{subequations}

\subsection{The curvatures}	\label{DWP curvatures}

Let uppercase Roman indices lie in $\{0,1,\dots,n_1+n_2\}$, lowercase Romans in $\{1,\dots,n_1\}$, and
lowercase Greeks in $\{n_1+1,\dots,n_1+n_2\}$.

We recall the standard formula
$R_{ABC}^D = \pd_A\Gamma_{BC}^D-\pd_B\Gamma_{AC}^D+\Gamma_{AE}^D\Gamma_{BC}^E-\Gamma_{BE}^D\Gamma_{AC}^E$.
In general, multiply-warped products have four flavors of sectional curvature corresponding to the following pairings of tangent planes:
(a) base with base, (b) fiber with itself, (c) fiber with distinct fiber, and (d) base with fiber.
\medskip

Because our base is one-dimensional,
only the last three flavors are relevant here, all of which may be obtained from the following formulas:
\begin{align*}
R_{ijk\ell}	&=g_{\ell m}R_{ijk}^m = v_1^2 \hat R_{ijk\ell}-\frac{v_{1,s}^2}{v_1^2}\Big((g_1)_{i\ell}(g_1)_{jk}-(g_1)_{ik}(g_1)_{j\ell}\Big),\\
R_{\alpha\beta\gamma\delta}	&=g_{\delta\eta}R_{\alpha\beta\gamma}^\eta = v_2^2 \hat R_{\alpha\beta\gamma\delta}
-\frac{v_{2,s}^2}{v_2^2}\Big((g_2)_{\alpha\delta}(g_2)_{\beta\gamma}-(g_2)_{\alpha\gamma}(g_2)_{\beta\delta}\Big),\\
R_{i\beta\gamma k} &=g_{k\ell}R_{i\beta\gamma}^\ell
=g_{k\ell}\Big(\pd_i\Gamma_{\beta\gamma}^\ell-\pd_\beta\Gamma_{i\gamma}^\ell
+\Gamma_{iE}^\ell\Gamma_{\beta\gamma}^E-\Gamma_{\beta E}^\ell\Gamma_{i\gamma}^E\Big)
=g_{k\ell}\Gamma_{i0}^\ell\Gamma_{\beta\gamma}^0\\
&=-\frac{v_{1,s}}{v_1}\frac{v_{2,s}}{v_2}(g_1)_{ik}(g_2)_{\beta\gamma},\\
R_{i00\ell}	&=g_{k\ell}R_{i00}^k=g_{k\ell}\Big(-\pd_s\Gamma_{i0}^k-\Gamma_{0p}^k\Gamma_{i0}^p\Big)=-\frac{v_{1,ss}}{v_1}(g_1)_{i\ell},\\
R_{\alpha00\delta}	&=g_{\delta\eta}R_{\alpha00}^\eta
=g_{\delta\eta}\Big(-\pd_s\Gamma_{\alpha0}^\eta-\Gamma_{0\ve}^\eta\Gamma_{\alpha0}^\ve\Big)=-\frac{v_{2,ss}}{v_2}(g_2)_{\alpha\delta}.
\end{align*}
\medskip

Now let $a,b\in\{1,2\}$. Then it follows from the formulas above that the Ricci curvature is determined by
\begin{subequations}		\label{Ricci}
\begin{align}
R_{00}	&= -\left(\sum_a  n_a\frac{(v_a)_{ss}}{v_a}\right)\mr ds^2,\\
(\Rc_a)_{\alpha\beta}	&=
\left\{-\frac{(v_a)_{ss}}{v_a}+(n_a-1)\frac{1-(v_a)_s^2}{v_a^2}-\sum_{b\neq a}n_b\frac{(v_a)_s(v_b)_s}{v_a v_b}\right\}(g_a)_{\alpha\beta}.
\end{align}
\end{subequations}

The scalar curvature is
\begin{equation}
R=-2\sum_a n_a\frac{(v_a)_{ss}}{v_a}
+\sum_a n_a(n_a-1)\frac{1-(v_a)_s^2}{v_a^2}
-\sum_a\sum_{b\neq a} n_a n_b\frac{(v_a)_s(v_b)_s}{v_a v_b}.
\end{equation}

% we don't use this anymore
%
% \subsection{The commutator}
% Using $s$ as a coordinate induces a commutator,
% \begin{equation}					\label{Commutator}
% \big[\partial_t,\,\partial_s\big]
% =\big[\partial_t,\,v_0^{-1}\partial x\big]
% =-\frac{v_{0,t}}{v_0^2}\,\partial_x=-\frac{v_{0,t}}{v_0}\,\partial_s.
% \end{equation}

\section{Details of the $\Delta h$ computation}		\label{Tensor Hessian}

In this appendix, we continue the notational conventions of Appendix~\ref{Geometric data} and compute $\cv^2_{AB} h_{CD}$ for a
metric~\eqref{eq:DWP Ansatz}, where uppercase Roman indices again range over all possible values in $\{0,1,\dots,n_1+n_2\}$.

Without loss of generality, we may work in normal coordinates for $\hat g_1$ and $\hat g_2$.

We start with $\cv_B h_{CD}$ and find that the only nonvanishing components are of the form
\begin{align*}
\cv_0 h_{00}= \eta_{0,s},  \qquad \cv_0h_{ij}=\eta_{1,s}g_{ij}, \qquad \cv_0 h_{\alpha\beta}=\eta_{2,s}g_{\alpha\beta},\\
\cv_i h_{j0}=\frac{v_{1,s}}{v_1}(\eta_0-\eta_1)g_{ij}, \qquad \cv_\alpha h_{\beta0}=\frac{v_{2,s}}{v_2}(\eta_0-\eta_2)g_{\alpha\beta}.
\end{align*}

We proceed to $\cv_A\cv_B h_{CD}$, noting that the only components of interest occur for compatible choices of $A,B$ and $C,D$.
For these, we obtain
\begin{align*}
\cv_0\cv_0 h_{00}	&=\eta_{0,ss},\\
\cv_i\cv_j h_{00}	&=\Big\{\frac{v_{1,s}}{v_1}\eta_{0,s}-2\frac{v_{1,s}^2}{v_1^2}(\eta_0-\eta_1)\Big\}g_{ij},\\
\cv_\alpha\cv_\beta h_{00}	&=\Big\{\frac{v_{2,s}}{v_2}\eta_{0,s}-2\frac{v_{2,s}^2}{v_2^2}(\eta_0-\eta_2)\Big\}g_{\alpha\beta},\\
\cv_0\cv_0 h_{ij}	&=\eta_{1,ss}\,g_{ij},\\
\cv_0\cv_0 h_{\alpha\beta}	&=\eta_{2,ss}\,g_{\alpha\beta},\\
\cv_i\cv_j h_{k\ell}	&=\frac{v_{1,s}}{v_1}\eta_{1,s}\,g_{ij}g_{k\ell}-\frac{v_{1,s}^2}{v_1^2}(\eta_1-\eta_0)\,(g_{ik}g_{j\ell}+g_{i\ell}g_{jk}),\\
\cv_\alpha\cv_\beta h_{\gamma\delta}	&=\frac{v_{2,s}}{v_2}\eta_{2,s}\,g_{\alpha\beta}g_{\gamma\delta}-
\frac{v_{2,s}^2}{v_2^2}(\eta_2-\eta_0)\,(g_{\alpha\gamma}g_{\beta\delta}+g_{\alpha\delta}g_{\beta\gamma}),\\
\cv_i\cv_j h_{\alpha\beta}	&=\frac{v_{1,s}}{v_1}\eta_{2,s}\,g_{ij}g_{\alpha\beta},\\
\cv_\alpha\cv_\beta h_{ij}	&=\frac{v_{2,s}}{v_2}\eta_{1,s}\,g_{ij}g_{\alpha\beta}.
\end{align*}
Tracing these identities yields the  formulas stated in \S~\ref{Lichnerowicz}.

\section{Invariant manifolds} \label{sec:invariant manifold construction}
In this appendix, we sketch the proof of Theorem~\ref{UnstableManifoldTheorem} adapted to our application,
using notation we introduced in \S~\ref{sec:Irwin Chaperon}.

Let $E$  be a Banach space, $\mc U\subset E$ an open subset, and $f:\mc U\to E$ a $C^r$ map with a fixed point $p\in\mc U$. 
Assume that $T=\mr df_p$ satisfies the $a$-hyperbolicity condition for some $a>1$.  We thus have $E=E^s\oplus E^u$, and 
\[
T = T^s\oplus T^u, \quad\mbox{with}\quad
T^s:E^s\to E^s \quad\mbox{and}\quad
T^u:E^u\to E^u.
\]
Assuming without loss of generality that $p=0_E$ and hence that $f(0_E)=0_E$, we expand $f$ in a Taylor series,
\[
f(x) = T x + g(x),
\]
where $g(0_E)=0_E$ and $\mr dg_0 = 0$.
The key step in the proof of Theorem~\ref{UnstableManifoldTheorem} is then to establish the following result:

\begin{lemma}
For every $v\in E^u$, there exists a unique ancient orbit $(x_i)_{i\leq 0}$ in $\mc U$ such that $\pi^u (x_0) = v$
and such that $\|x_i\|\leq Ca^i$ for some $C>0$ and all $i\leq 0$.
\end{lemma}

In this Lemma the map $f$ is not assumed to be a local diffeomorphism or even injective.  The only assumptions are that $f$ be $C^1$ in
the sense of Fréchet derivatives, and that the spectrum $df_p$ admit a splitting as described above.

\begin{proof}[Sketch of proof]
If $\varepsilon>0$ is sufficiently small, then the spectrum $\sigma(T)$ does not intersect the annulus
\[
\{\lambda\in\C : a-\varepsilon \leq |\lambda| \leq a+\varepsilon\}.
\]
We may replace the norm on $E$ by an equivalent norm in which
\[
\forall v\in E^s: \|T v\| \leq (a-\varepsilon)\|v\|	\qquad\mbox{and}\qquad
\forall v\in E^u: \|T v\| \geq (a+\varepsilon)\|v\|. 
\]

Split $x_i = x^s_i\oplus x^u_i$.  Then the equation $x_i=f(x_{i-1})$ is equivalent to the pair
\begin{equation}
\label{eq:recurrence stable and unstable components}
x^s_i = T  x^s_{i-1} + g^s(x_{i-1}) \qquad\mbox{and}\qquad
x^u_i = T  x^u_{i-1} + g^u(x_{i-1}) .
\end{equation}
The first recurrence equation implies that
\[
x^s_i = T^N x^s_{i-N} + g^s(x^s_{i-1})+ Tg^s(x^s_{i-2})+ \cdots + T^{N-1}g^s(x^s_{i-N}).
\]
If we only consider sequences $x_i$ for which $\|x_i\|\leq Ca^i$ for some constant $C>0$, then their stable components also decay
like $a^i$, and thus we have
\[
\|T^Nx^s_{i-N}\| \leq (a-\varepsilon)^N C a^{i-N} = Ca^i \Bigl(1-\frac{\varepsilon}{a}\Bigr)^N \to 0 \quad\mbox{as}\quad N\to\infty.
\]
The first equation in \eqref{eq:recurrence stable and unstable components} is therefore equivalent to
\begin{equation}        \label{eq:convolution stable component}
x^s_i = \sum_{j=0}^\infty  T^j  g^s(x^s_{i-j-1}).
\end{equation}

We can rewrite the second recurrence equation in \eqref{eq:recurrence stable and unstable components} in a similar manner if we
use the fact that that $T|E^u$ is invertible --- this holds because $\sigma(T)$ lies outside the circle with radius $a+\varepsilon$,
which implies that $0\not\in\sigma(T)$.  Thus we have
\[
x^u_i = T^{-1} x^u_{i+1} - T^{-1}g^u(x_{i+1}).
\]
Iterating this recurrence $i$ times, we find that for any $i\leq 0$,
\begin{equation}
\label{eq:convolution unstable component}
x^u_i = T^{-i} x^u_0 - \sum_{j=1}^i T^{-j}g^u(x_{i+j}).
\end{equation}

We are looking for orbits $x_i$ for which $\pi^u x_0 = x^u_0 = v$.  The set of orbits we are interested in is therefore determined by the pair of equations
\begin{equation} \label{eq:unstable mfd prepped for IFThm}
\begin{aligned}
x^s_i -\sum_{j=0}^\infty  T^j  g^s(x^s_{i-j-1})  &= 0,\\
x^u_i + \sum_{j=1}^i T^{-j}g^u(x_{i+j})&=  T^{-i} v.
\end{aligned}
\end{equation}
Because of the decay condition that we impose on the orbits, we are looking for solutions $(x_i)_{i\leq 0}$ to these equations in the sequence space
\[
\ell^\infty_a(E) \stackrel{\rm def}=
\bigl\{  (x_i)_{i\leq 0} : \|(x_i)\|_a <\infty
\bigr\}, \quad\mbox{where}\quad
\|(x_i)\|_a \stackrel{\rm def}= \sup_{i\leq 0} a^{-i}\|x_i\|.
\]

To complete the proof, one verifies that the map $\Upsilon : (x_i)_{i\leq 0} \mapsto (y_i)_{i\leq 0}$ given by  
\[
y_i =  \left(-\sum_{j=0}^\infty T^j  g^s(x^s_{i-j-1})\right)  \oplus \left(\sum_{j=1}^i T^{-j}g^u(x_{i+j})\right)
\]
is a continuously Fréchet differentiable map $\Upsilon\colon\ell^\infty_a(E)\to\ell^\infty_a(E)$ with derivative $\mr d \Upsilon_0 = 0$.
One can then apply the Implicit Function Theorem to the equations~\eqref{eq:unstable mfd prepped for IFThm} to conclude that a
unique solution $(x_i)_{i\leq 0}\in \ell^\infty_a(E) $ exists for any small prescribed $v\in E^u$.

\end{proof}

\section{Bounded oscillations}		\label{OscillationBound}

In this Appendix we record an observation, which we do not use in this paper but which might be of independent interest.

We consider a doubly warped product $\mfd = \mb R^{n_1+1} \times \mc S^{n_2}$ evolving by the nonlinear
Ricci--DeTurck flow.  We specialize to Ricci--DeTurck flow of $\mf G = \mr{O}(n_1+1)\times\mr{O}(n_2+1)$-invariant metrics.
Accordingly, we proceed to consider the evolution of
\begin{equation} \label{G-invariant metric}
\gh = \mr ds^2+g_1+g_2 \stackrel{\rm def} = v_0^2(\dx^2) + v_1^2\,\hat g_1+v_2^2\,\hat g_2
\end{equation}
by Ricci--DeTruck flow~\eqref{eq:RDT}.  As background metric, we choose one of the Ricci-flat Böhm metrics we considered in \cite{GAFA},
which we write in this appendix as
\begin{equation}\label{eq:bohm}
g_{\mr B} = (\mr dx)^2 + V_1(x)^2\hat g_1 + V_2(x)^2 \hat g_2(x).
\end{equation}

Our objective here is to prove:

\begin{lemma}	\label{Winding}
If $\big(\mfd=\mb R^{1+n_1}\times\mc S^{n_2},\tilde g = v_0^2\,\dx^2 + v_1^2\,\hat g_1+v_2^2\,\hat g_2\big)$ is a doubly warped product
manifold evolving by Ricci--DeTurck flow~\eqref{eq:RDT}, then the number of $x>0$ for which
\[
  \frac{v_1(x, t)}{\sqrt{n_1-1}} = \frac{v_2(x, t)}{\sqrt{n_2-1}}
\]
does not increase in time.
\end{lemma}

\subsection{Regularity of the flow near a Böhm soliton}

Our short-time existence theorem provides solutions $\tilde g(t, \cdot) = g_{\mr B} + h$ where $\epsilon g_{\mr B} < \tilde g < \epsilon^{-1} g_{\mr B}$
for some $\epsilon>0$.  Metrics that start out as doubly warped products remain doubly warped products.  We will therefore consider solutions where $h$
has  the form
\[
  h = h_0(x, t)(\mr dx)^2 + h_1(x, t)\hat g_1 + h_2(x, t)\hat g_2,
\]
with
\[
h_0 = v_0^2- 1, \qquad
h_1 = v_1^2- V_1^2, \qquad
h_2 = v_2^2- V_2^2.
\]

\subsection{The DeTurck vector field}

As in Appendix~\ref{Geometric data}, we let uppercase Roman indices lie in $\{0,1,\dots,n_1+n_2\}$, lowercase Romans in $\{1,\dots,n_1\}$,
and lowercase Greeks in $\{n_1+1,\dots,n_1+n_2\}$. Here writing $\tilde g^{AB} = (\tilde g^{-1})^{AB}$, we apply~\eqref{eqn:slick X} to see
that the DeTurck vector field $X$ has components
\begin{align*}
X^A	&=\tilde g^{BC}(\tilde\Gamma_{BC}^A-\Gamma_{BC}^A)\\
&=\tilde g^{00}(\tilde\Gamma^A_{00}-\Gamma^A_{00})+\tilde g^{ij}(\tilde\Gamma^A_{ij}-\Gamma^A_{ij})
+\tilde g^{\alpha\beta}(\tilde\Gamma^A_{\alpha\beta}-\Gamma^A_{\alpha\beta}).
\end{align*}
By the formulas in Appendix~\ref{DWP connection}, this vanishes unless $A=0$, whereupon we obtain $X=f(s)\,\frac{\partial}{\partial s}$,
where
\begin{equation}	\label{DWP DeTurck}
f(s) 	= v_{0,s}-n_1\frac{v_{1,s}}{v_1}-n_2 \frac{v_{2,s}}{v_2}
	= \frac{v_{0,x}}{v_0}-n_1\frac{v_{1,s}}{v_1}-n_2 \frac{v_{2,s}}{v_2}.
\end{equation}

\subsection{A simplification of the system}

In formula~\eqref{Ricci} of Appendix~\ref{DWP curvatures}, we show that
\begin{align*}
-\widetilde\Rc=& \left\{\sum_a n_a \frac{v_{a,ss}}{v_a}\right\}\;v_0^2\, (\dx)^2\\
&+\sum_a\left\{\Bigl(\frac{v_{a,s}}{v_a}\Bigr)_s   - \frac{n_a-1}{v_a^2}+ \frac{v_{a,s}}{v_a}\sum_b n_b\frac{v_{b, s}}{v_b}\right\} v_a^2 \, \hat g_a,
\end{align*}
where $a,b\in\{1,2\}$.  As shown above, the Lie derivative we need is
\[
\mc L_{X} \tilde g = 2f_s(\mr d s)^2+2\Big(\sum_a f\frac{v_{a,s}}{v_a}\Big)\, v_a^2\,\hat g_a.
\]

It follows that Ricci--DeTurck flow~\eqref{eq:RDT} is equivalent to the strictly parabolic system of equations
\begin{align*}
v_0 \frac{\pd v_0}{\pd t}
&= \Bigl\{ \sum_a n_a\frac{v_{a,ss}}{v_a}+f_s\Bigr\}\,v_0^2,\\
v_a  \frac{\pd v_a}{\pd t}
% &=  \left\{\Bigl(\frac{v_{a,s}}{v_a}\Bigr)_s   - \frac{n_a-1}{v_a^2}+ \frac{v_{a,s}}{v_a}\sum_b n_b \frac{v_{b, s}}{v_b}+ f \frac{v_{a,s}}{v_a}\right\}\,v_a^2\\
&=  \left\{\Bigl(\frac{v_{a,s}}{v_a}\Bigr)_s   - \frac{n_a-1}{v_a^2}+ \frac{v_{a,s}}{v_a}\left(f+\sum_b n_b \frac{v_{b, s}}{v_b}\right)\right\}v_a^2,
\end{align*}  
which in turn are the same as
\begin{align*}
\frac{\pd v_0}{\pd t}	&= \Bigl\{f_s + \sum_a n_a\frac{v_{a,ss}}{v_a}\Bigr\}\,v_0,\\
\frac{\pd v_a}{\pd t}		&=  \left\{\Bigl(\frac{v_{a,s}}{v_a}\Bigr)_s   - \frac{n_a-1}{v_a^2}+ \frac{v_{a,s}}{v_a}\left(f+\sum_b n_b\frac{v_{b, s}}{v_b}\right)\right\}v_a.
\end{align*}
\medskip

This system becomes even simpler if one writes it in terms of
\[
w_0\stackrel{\rm def}	= \log (v_0)\quad\mbox{and the previously defined}\quad
w_a	= \log \Big(\frac{v_a}{\sqrt{n_a-1}}\Big),\quad(a=1,2).
\]
Then~\eqref{DWP DeTurck} implies that $f=w_{0,x}-n_1w_{1,s}-n_2w_{2,s}$, and one obtains
\begin{subequations}		\label{eqn:slick system}
\begin{align}
\frac{\pd w_0}{\pd t}	&= \bigl(f+\sum_a n_a w_{a,s}\bigr)_s + \sum_a n_a w_{a,s}^2,  \\
\frac{\pd w_a}{\pd t}	&= w_{a, ss} + \bigl(f+\sum_b n_b w_{b,s}\bigr) w_{a,s} -  e^{-2w_a},
\end{align}
\end{subequations}
in which both $w_a$ and $w_b$ satisfy identical equations.
\medskip

We note that the difference $w_{12} \stackrel{\rm def}= w_1-w_2$ satisfies a linear parabolic equation,
\[
\frac{\pd}{\pd t} w_{12}
= a(x,t)\,w_{12, xx} + b(x,t)\, w_{12,x} + c(x,t)\,w_{12}
\]
where the relation $s=v_0\,x$ implies that
\[
a(x,t) = v_0^{-2},\quad
b(x,t) = v_0^{-1}\Big\{f+\sum_a n_a w_{a,s}-v_0^{-2}v_{0,s}\Big\}, \quad
c(x,t) = \frac{e^{-2w_2} - e^{-2w_1}}{w_1-w_2}.
\]

\begin{proof}[Proof of Lemma~\ref{Winding}]

We wish to apply the Sturmian theorem~\cite{Sturmian}. We note that $v_0>0$ is a necessary condition for the solution to exist.
Therefore, on any time interval $[0,T]$ on which the solution remains smooth, one verifies from~\eqref{eqn:slick system} that
$a,a_x,a_{xx},a_t,a^{-1}\in L_\infty$ and that $b,b_x,b_t\in L_\infty$. Then writing
\[
c=e^{-2w_2}\Big(\frac{1-e^{-2w_{12}}}{w_{12}}\Big),
\]
one readily sees that $c\in L_\infty$ as well. The necessary conditions for the metric to close smoothly at the boundary $s=0$ imply that
$w_1\to-\infty$ as $s\searrow0$, while $w_2$ remains bounded. This implies that no zeroes of $w_{12}$ come in from the boundary.
The result then follows directly from~\cite{Sturmian}.
\end{proof}

\section{Function spaces}			\label{sec:function spaces}

In this appendix, we review properties of little-Hölder spaces on Riemannian manifolds of bounded geometry.  For details, we refer
to Aubin's book~\cite{TAubin98}, as well as papers by Hans Triebel~\cite{Triebel86,Triebel87}.\footnote{Note that Triebel treats the
  Besov spaces $B^s_{pq}(\mfd )$ of which the Hölder spaces are a special case: $h^{k,\alpha} = B^{k+\alpha}_{\infty\infty}$,
  provided that $\alpha\in(0, 1)$ and $k\in\Z$.}

\subsection{A covering lemma}\label{lem:covering}

\begin{definition}
The \emph{dimension} of a covering $\ball_a$ ($a\in\N$) is the smallest $N\in\N$ such that $\ball_{a_0}\cap \ball_{a_1}\cap \cdots \cap \ball_{a_N}=\varnothing$ for all $a_0<a_1<\cdots <a_N\in\N$.  
\end{definition}

\begin{lemma}
Suppose $(\mfd , g)$ has bounded geometry.  Then there are $N\in\N$ and $r_0>0$ such that for every $r\in(0, r_0)$, a sequence
$\mf p_a\in\mfd $ exists for which $\{\ball_a:=\ball(\mf p_a, r)\}_{a\in\N}$ is an open cover of $\mfd $ of dimension at most $N$.
\end{lemma}

This result appears in \cite{Triebel86,Triebel87}.
A proof is indicated in \cite[Lemma 2.26]{TAubin98}. The reader may also wish to consult~\cite{Hebey}.

The idea behind the proof is to use the Vitali covering lemma to find a sequence $\{\mf p_a\}_{a\in\N}$ such that $\mfd = \bigcup_a \ball(\mf p_a, r)$, while
the smaller balls $\ball(\mf p_{a}, \frac15 r)$ are pairwise disjoint.  The curvature bounds on the metric imply that for small enough $r>0$
one has
\[
\mathrm{vol}\bigl(\ball(\mf p, \tfrac15 r)\bigr) \geq cr^n  \qquad\mbox{ and }\qquad \mathrm{vol}\bigl(\ball(\mf p, 3 r)\bigr) \leq Cr^n
\]
for every $p\in\mfd $, where $c, C>0$ do not depend on $r$ or $p$.

If
\[
\ball(\mf p_{a_0}, r)\cap \cdots \cap \ball(\mf p_{a_N}, r) \neq \varnothing,
\]
then $\mf p_{a_1}, \dots, \mf p_{a_N}\in \ball(\mf p_{a_0}, 2r)$, and thus the disjoint balls $\ball(\mf p_{a_0}, \frac15 r)$, \dots,
$\ball(\mf p_{a_N}, \frac15 r)$ are subsets of $\ball(\mf p_{a_0}, 3r)$.

Comparing the volumes of $\cup_{i=0}^N \ball(\mf p_{a_i}, \frac15 r)$ and $\ball(\mf p_{a_0}, 3r)$ we find that 
\[
(N+1) cr^n \leq Cr^n.
\]
This is a contradiction if $N \geq C/c$, so that the intersection of more than $C/c+1$ balls $\ball(\mf p_a, r)$ is always empty.

\subsection{Partition of unity adapted to a covering}\label{sec:partition of unity}
Let $\ball_a:= \ball(\mf p_a, r)$ be an open cover of a manifold $(\mfd , g)$ of bounded geometry, as in Lemma~\ref{lem:covering}.  Then
there exists a set $\{\varphi_a\in C_c^\infty(\mfd )\colon a\in\mr N\}$ with $\mathop{\mathrm{supp}}\varphi_a= \ball_a$ such that
\[
\sum_a \varphi_a(x)^2 = 1
\]
for all $x\in\mfd $.

Furthermore, we can choose the $\varphi_a$ so that
\begin{equation}\label{eq:partition-gradients}
r|\nabla\varphi_a|+r^2|\nabla\nabla\varphi_a| + r^3|\nabla\nabla\nabla\varphi_a|\leq C
\end{equation}
for all small $r>0$, and for some constant $C$ that does not depend on $r$.

Indeed, choose $\psi\in C^\infty(\R)$ with $\psi(t) = 0$ for $t\geq 1$ and $\psi(t)>0$ for $t<1$.  Then let
\[
\phi_a(x) = \psi\bigl({d(x, \mf p_a)^2}/{r^2}\bigr) \text{ for }x\in B, \qquad \phi_a(x) = 0 \text{ elsewhere,}
\]
and define
\[
\varphi_a(x) = \frac{\phi_a(x)}{\sqrt{\sum_b \phi_b(x)^2}}.
\]

We can estimate the gradient of $\phi_a$ using
\[
|\nabla d(x, p)^2| = 2d(x, p),
\]
which implies that
\[
\nabla\phi_a = \frac{\psi'(d^2/r^2)}{r^2} \nabla(d^2) \quad\implies\quad |\nabla\phi_a| = \frac{2d}{r^2}|\psi'(d^2/r^2)| \le \frac Cr.
\]
For the Hessian, we have
\[
\nabla\nabla\phi_a =\frac{\psi''(d^2/r^2)}{r^4} \nabla(d^2)\otimes \nabla(d^2) + \frac{\psi'(d^2/r^2)}{r^2} \nabla\nabla(d^2)
\quad\implies\quad |\nabla\nabla\phi_a|\le \frac{C}{r^2}.
\]
Reasoning similarly, one sees that
\[
|\nabla^3 \phi_a|\le \frac{C}{r^3}.
\]

\subsection{Little-Hölder spaces}\label{About little Holder}
If $(K, d)$ is a compact metric space, one says that a function $f:K\to \R$ is
\emph{Hölder continuous of exponent $\alpha$} if
\[
\sup_{x\neq y} \frac{|f(x)-f(y)|}{d(x,y)^\alpha}<\infty.
\]
The set of $\alpha$-Hölder continuous functions on $K$ becomes a Banach space if one gives it the norm
\[
\|f\|_{0,\alpha} := \sup_x |f(x)| + \sup_{x\neq y} \frac{|f(x)-f(y)|}{d(x,y)^\alpha}.
\]

Recall that smooth functions are \emph{not dense} in $C^{0,\alpha}$. For example, one can verify that the distance in
$C^{0,1/2}([-1,1])$ between $f(x)=\sqrt{|x|}$ and any Lipschitz function $g:[-1,1]\to\R$ is at least $1$.

A function $f:K\to\R$ is \emph{little-Hölder continuous} if
\[
\lim_{\epsilon\searrow 0}\epsilon^{-\alpha} \sup\bigl\{|f(x)-f(y)| : d(x, y)<\epsilon\bigr\} =0.
\]
We write $h^{0,\alpha}(K)$ for the set of little-Hölder continuous functions on $K$. $h^{0,\alpha}(K)$ is a closed
subspace of $C^{0,\alpha}(K)$.  If $K=\ball(0, R)\subset\R^n$, then it is the closure of $C^\infty(K)$ in
$C^{0,\alpha}(K)$.

\subsection{Little-Hölder sections of a tensor bundle}
Let $(\mfd , g)$ be a manifold with bounded geometry, and let $T^{(p,q)}\mfd \to\mfd $ be the bundle of $(p,q)$ tensors
on $\mfd $. For simplicity of notation we describe the case of $(0,2)$-tensors (such as the metric and Ricci tensor).

Consider a covering $\ball_a:= \ball(\mf p_a, r)$ as in Lemma~\ref{lem:covering}, and parameterize each $\ball_a$ by the exponential map at
$\mf p_a$, \emph{i.e.,} by $\psi_a:\mc B\to {\ball(\mf p_a, r)}$ given by $\psi_a(\xi) = \exp\mf p_{p_a}(r\xi)$.  We say that a
$(0,2)$-tensor $T$ is \emph{Hölder continuous of exponent $\alpha$} if its components $T_{jk}$ (defined by
$T(\xi)=T_{jk}(\xi)\,\mr d \xi^j\,\mr d\xi^k$) are Hölder continuous functions in all of the coordinate patches $\ball_a$, and if
\[
\|T\|_{0,\alpha} \stackrel{\rm def}= \sup_a\max_{j,k} \|T_{jk}(\xi)\|_{h^{0,\alpha}(\ball_a)} <\infty.
\]
The quantity $\|T\|_{0,\alpha}$ is the $\alpha$-Hölder norm.
Changing the cover $\{\ball_a\}$ leads to a different but equivalent norm, provided that $r>0$ is small enough.

\bibliographystyle{plain} \bibliography{refs.bib}

\end{document}